\documentclass[10pt,reqno]{amsart}
\usepackage{longtable}
\usepackage{hyperref}
\usepackage[T1]{fontenc}
\usepackage[utf8]{inputenc}
\usepackage[english]{babel}
\usepackage{textcomp}
\usepackage{dsfont}
\usepackage{latexsym}
\usepackage{amssymb}
\usepackage{amsthm}
\usepackage{amsmath}
\DeclareMathAlphabet{\mathpzc}{OT1}{pzc}{m}{en}
\usepackage{yfonts}
\usepackage{xfrac}
\usepackage{newlfont}
\usepackage{graphicx}
\usepackage{mathtools}
\usepackage{comment}
\usepackage{indentfirst}
\usepackage{braket}
\usepackage{mathrsfs}
\usepackage{xcolor}
\usepackage{fixmath}

\numberwithin{equation}{section}

\usepackage{scalerel}[2014/03/10]
\usepackage[usestackEOL]{stackengine}
\newcommand{\dashint}{\,\ThisStyle{\ensurestackMath{%
\stackinset{c}{.2\LMpt}{c}{.5\LMpt}{\SavedStyle-}{\SavedStyle\phantom{\int}}}%
\setbox0=\hbox{$\SavedStyle\int\,$}\kern-\wd0}\int}

\DeclareMathOperator{\pr}{pr}

\DeclareMathOperator{\supp}{Supp}

\DeclareMathOperator{\tr}{Tr}

\DeclareMathOperator{\Hol}{Hol}

\DeclareMathOperator{\Sp}{sp}

\renewcommand{\Re}{\mathrm{Re}\,}
\renewcommand{\Im}{\mathrm{Im}\,}
\newcommand{\Supp}[1]{\supp\left( #1\right) }
\newcommand{\Pfaff}{\mathrm{Pf}}
\newcommand{\ee}{\mathrm{e}}

\newcommand{\vect}[1]{\mathbf{{#1}}}
\newcommand{\dd}{\mathrm{d}}

\DeclarePairedDelimiter{\abs}{\lvert}{\rvert}

\DeclarePairedDelimiter{\norm}{\lVert}{\rVert}

\let\originalleft\left
\let\originalright\right
\renewcommand{\left}{\mathopen{}\mathclose\bgroup\originalleft}
\renewcommand{\right}{\aftergroup\egroup\originalright}

\newcommand{\N}{\mathds{N}}
\newcommand{\Z}{\mathds{Z}}

\newcommand{\C}{\mathds{C}}

\newcommand{\R}{\mathds{R}}

\newcommand{\Sf}{\mathfrak{S}}

\newcommand{\Bc}{\mathcal{B}}

\newcommand{\Ec}{\mathcal{E}}
\newcommand{\Fc}{\mathcal{F}}

\newcommand{\Kc}{\mathcal{K}}
\newcommand{\Lc}{\mathcal{L}}
\newcommand{\cM}{\mathcal{M}}  
\newcommand{\Nc}{\mathcal{N}}
\newcommand{\Oc}{\mathcal{O}}

\newcommand{\Sc}{\mathcal{S}}

\newcommand{\sE}{\mathscr {E}}
\newcommand{\sH}{\mathscr{H}}

\newcommand{\meg}{\leqslant}
\newcommand{\Meg}{\geqslant}
\newcommand{\eps}{\varepsilon}
\renewcommand{\phi}{\varphi}

\newcommand{\Lin}{\mathscr{L}}

\newcommand{\la}{\langle}
\newcommand{\ra}{\rangle}
\newcommand{\sinc}{\operatorname{sinc}}

\newcommand{\ov}{\overline}

\title{Bernstein Spaces on Siegel CR Manifolds}

\date{\today}
\begin{document}

\theoremstyle{definition}
\newtheorem{deff}{Definition}[section]

\newtheorem{oss}[deff]{Remark}
\newtheorem{ass}[deff]{Assumptions}
\newtheorem{nott}[deff]{Notation}
\newtheorem{example}[deff]{Example}

\theoremstyle{plain}
\newtheorem{teo}[deff]{Theorem}
\newtheorem{lem}[deff]{Lemma}
\newtheorem{prop}[deff]{Proposition}
\newtheorem{cor}[deff]{Corollary}
\author[M. Calzi, M. M. Peloso]{Mattia Calzi, 
Marco M. Peloso}

\address{Dipartimento di Matematica, Universit\`a degli Studi di
  Milano, Via C. Saldini 50, 20133 Milano, Italy}
\email{{\tt mattia.calzi@unimi.it}}
\email{{\tt marco.peloso@unimi.it}}

\thanks{{\em Math Subject Classification 2020} Primary: 32A15, 32A37. Secondary: 32A50, 46E22.}
\thanks{The authors are members of the
  Gruppo Nazionale per l'Analisi Matematica, la Probabilit\`a e le
  loro 
Applicazioni (GNAMPA) of the Istituto Nazionale di Alta Matematica (INdAM) }
\thanks{The authors are partially supported by the 2020
  INdAM--GNAMPA grant
  {\em Fractional Laplacians and subLaplacians on Lie groups and trees}.}
\thanks{Declarations of interest: none.}
 \begin{abstract}
In this paper we introduce and study Bernstein spaces on a class of
quadratic CR manifolds, which we call {\em Siegel CR  manifolds}.
These are spaces of entire functions of exponential type whose restrictions to a given Siegel CR submanifold are $L^p$-integrable with respect to a natural measure.  For these spaces, among other results, we prove the Plancherel--P\'olya inequality, a Bernstein inequality and a sufficient condition for a sequence to be sampling. 
\end{abstract}
\maketitle

\section*{Introduction} 

Given $\kappa>0$, consider the space $\Ec_\kappa(\C)$ of entire function of
exponential type at most $\kappa$, that is,
\[
\Ec_\kappa(\C)=\Set{ f\in\Hol(\C)\colon \limsup_{w \to \infty} \frac{\log\abs{f(w)}}{\abs{w}}\meg \kappa}. 
\]
In other words, $f\in \Ec_\kappa(\C)$ if and only if for every $\eps>0$ there is a constant $C_\eps>0$ such that
\[
  \abs{f(w)}\le C_\eps \ee^{(\kappa+\eps)\abs{w}}
\]
for every $w\in \C$. Given $p\in (0,\infty]$, the classical Bernstein space $\Bc^p_\kappa(\C)$ is defined as the space of functions in $\Ec_\kappa(\C)$ whose restrictions to the real line are in $L^p(\R)$.  In other words, setting  $f_v\colon u\mapsto f(u+iv)$ for every $v\in \R$, 
\[
\Bc_\kappa^p =\Set{ f\in\Ec_\kappa(\C)\colon f_0 \in L^p(\R) },
\]
endowed with the quasi-norm 
\[
\norm{f}_{\Bc_\kappa^p}\coloneqq\norm{f_0}_{L^p}
\]
for every $f\in \Bc_\kappa^p$.
Notice that $\Bc_\kappa^p\subseteq \Bc_\kappa^\infty$ (cf.~\cite[No.\ 30]{PlancherelPolya}), so that the Phragm\'en--Lindel\"of principle implies
that  $\abs{f(w)}\le \norm{f_0}_{L^\infty}\ee^{\kappa\abs{w}}$ for every $w\in\C$ and for every $f\in \Bc^p_\kappa$. Consequently, \emph{a posteriori}, $f$ satisfies more precise growth conditions than those initially assumed.

We shall now review some basic facts on the spaces $\Bc^p_\kappa$ which are particularly relevant to this work. Namely, we shall consider: Paley--Wiener theorems; Plancherel--P\'olya inequalities;  Bernstein inequalities; sampling sequences.

We begin with the characterization of $\Bc^p_\kappa$ by means of the Fourier transform (`Paley--Wiener theorems'). As usual, it is more convenient to start with the case $p=2$. 
In this case, the space $\Bc^2_\kappa$ is the classical Paley--Wiener space $PW_\kappa$; in fact, all Bernstein spaces are sometimes reffered to as Paley--Wiener spaces, even though we shall not follow this convention.  
The space $\Bc^2_\kappa$ is particularly simple to describe in terms of the Fourier transform, since the mapping $f \mapsto \frac{1}{\sqrt{2\pi}}\widehat{f_0}$ induces an isometric isomorphism of $\Bc^2_\kappa$ onto $L^2([-\kappa,\kappa])$.
In other words:
\begin{itemize}
\item[{\tiny $\bullet$}] if $f\in \Bc^2_\kappa$, then $\supp(\widehat{f_0})\subseteq  [-\kappa,\kappa]$ and $\norm{f}_{\Bc^2_\kappa} =\frac{1}{\sqrt{2\pi}} \norm{\widehat{f_0}}_{L^2}$;

\item[{\tiny $\bullet$}] if $g\in L^2([-\kappa,\kappa])$, the function $f$ on $\C$ defined by
\[
f(z) =\frac{1}{2\pi} \int_{-\kappa}^\kappa \ee^{iz\xi} g(\xi)\, \dd\xi,
\]
for every $z\in \C$, belongs to $\Bc^2_\kappa$ and satisfies $\norm{f}_{\Bc^2_\kappa}=\frac{1}{\sqrt{2\pi}}\norm{g}_{L^2}$. 
\end{itemize}
Then, $\Bc^2_\kappa$ is a reproducing kernel hilbertian space and  its reproducing kernel is the $\sinc$-function:
\[
\C\times\C\ni (z,w) \mapsto  \frac1\pi   \frac{\sin(\kappa(z-\ov w))}{z-\ov w} \eqqcolon
\frac\kappa\pi \sinc (\kappa(z-\ov w)).  
\]

The preceding characterization of $\Bc^p_\kappa$, $p=2$, by means of the Fourier transform may then be extended to all $p\in [1,\infty]$ (and also to the case $p\in (0,1)$ with some more care,  cf.~Theorem~\ref{cor:1}).  In fact, given $p\in [1,\infty]$, a function $g\in
L^p(\R)$ is the restriction to $\R$ of a function $f\in\Bc_\kappa^p$ if and only if the Fourier transform of $g$ (in the sense of distributions) is supported in $[-\kappa,\kappa]$, see e.g.~\cite[Theorem 4]{Andersen}.

Let us now pass to Plancherel--P\'olya inequalities (cf.~\cite[No.\ 27]{PlancherelPolya}).  Given $p\in(0,\infty]$ and  $f\in\Bc^p_\kappa$,
\[
\int_{-\infty}^\infty |f_y(x)|^p \, \dd x \le \ee^{p\kappa|y|} \norm{f}_{\Bc^p_\kappa}^p
\]
for every $y\in \R$, with the obvious modification when $p=\infty$.
As a consequence, $\Bc^p_\kappa$ is complete, so that it is a Banach space when $p\in[1,\infty]$, and a quasi-Banach space when $p\in (0,1)$.

Another important property of Bernstein spaces is
the renowned {\em Bernstein inequality} and the associated
characterization, see, e.g.,~\cite{Andersen}.  Given $p\in[1,\infty]$, the following hold:
\begin{itemize}
\item[{\tiny $\bullet$}] if $f\in \Bc^p_\kappa$, then $f^{(j)}\in
\Bc^p_\kappa$ for all $j\in\N$
and $\|f^{(j)}\|_{\Bc^p_\kappa} \le \kappa^j \|f\|_{\Bc^p_\kappa}$;
\item[{\tiny $\bullet$}] conversely, if $g\in C^\infty(\R) \cap L^p(\R)$ and 
$\|g^{(j)}\|_{L^p} \le \kappa^j \|g\|_{L^p}$ for
all $j\in\N$, then $g=f_0$ for a (unique) $f\in \Bc^p_\kappa$.
\end{itemize}

Concerning sampling sequences, by~\cite[Nos.\ 40, 44]{PlancherelPolya}, for every $p\in(0,\infty]$ and $\kappa'>\kappa$ there exist two constants $C_{p,\kappa,\kappa'},C'_{p,\kappa,\kappa'}>0$ such that, for every $f\in\Bc^p_\kappa$,
\[
C_{p,\kappa,\kappa'}\norm{ f}_{\Bc^p_\kappa} \le \Big( \sum_{n\in\Z}
\abs{f(n\pi/\kappa')}^p \Big)^{1/p} \le C'_{p,\kappa,\kappa'} \norm{f}_{\Bc^p_\kappa} .
\]
We observe explicitly that also the preceding inequalities are sometimes referred to as the Plancherel--P\'olya inequalities.
When $p\in (1,\infty)$, one may take also $\kappa'=\kappa$, while when $p=2$, the classical Whittaker--Kotelnikov--Shannon theorem  implies that
$C_{2,\kappa,\kappa}=C'_{2,\kappa,\kappa}=\sqrt{\kappa/\pi}$, so that the norm of $f$ may be exactly reconstructed from its samples.
General samplings sequences for $\Bc^2_\kappa$ have also been characterized (cf.,~\cite{Beurling} for sampling sequences in $\R$, and~\cite[Theorem 10 of Chapter 6]{Seip} for sampling sequences in $\C$). 

For these and other properties of the Bernstein spaces in one complex
variable, see, e.g.,~\cite{Levin-Lectures,Young}, and~\cite{MonguzziPelosoSalvatori-fractional}
for a generalization.

\medskip

Bernstein spaces have also been investigated in the context of several complex variables. In this case though, several notions of exponential type for an entire function may be employed. As a first choice, one may say that an entire function $f$ on $\C^m$ is of exponential type at most $\kappa$ if (cf., e.g.,~\cite{PlancherelPolya})
\[
\limsup_{z\to \infty} \frac{\log\abs{f(z)}}{\sum_j \abs{z_j}}\meg \kappa.
\]
Nontheless, further insight into the resulting Bernstein spaces may be obtained if one replaces $z\mapsto \sum_j\abs{z_j}$ with an arbitrary norm $\abs{\,\cdot\,}'$ on $\R^m$, extended to $\C^m$  so that
\[
\abs{z}'=\max_{\alpha\in \C, \abs{\alpha}=1} \abs{\Re (\alpha z)}' 
\]
for every $z\in \C^m$ (cf.~\cite[Chapter III, \S\ 4]{SteinWeiss}). Notice that,  the unit ball $B\coloneqq\Set{x\in \R^m\colon
  \abs{x}'\meg 1}$ uniquely determines the unit ball in the dual
norm  (which is also its polar set),
\[
B^\circ=\Set{\lambda\in \R^m\colon \forall x\in B\:\: \abs{\langle
    \lambda, x\rangle} \meg 1}, 
\]
which is a symmetric compact convex neighbourhood of $0$ in $\R^m$
(a `convex body'). Conversely, given a convex body $K$ in $\R^m$, the
Minkowski functional (or gauge) associated with its polar 
\[
K^\circ =\Set{x\in \R^m\colon \forall \lambda\in K\:\:\abs{\langle \lambda,x\rangle}\meg 1}
\]
is the norm given by
\[
\abs{x}_{K^\circ}=\sup_{\lambda\in K} \abs{\langle \lambda, x \rangle}
\]
and coincides with the supporting function of $K$, 
see~\eqref{supp-func:def}.   One may then define the space of entire functions of type $K$ as
\[
\Ec_K(\C^m)\coloneqq \Set{f\in \Hol(\C^m)\colon \limsup_{z\to \infty} \frac{\log \abs{f(z)}}{\abs{z}_{K^\circ}}\meg 1},
\]
so that $\Ec_\kappa(\C)=\Ec_{[-\kappa,\kappa]}(\C)$ for every $\kappa>0$. The corresponding Bernstein space of type $L^p$, $p\in (0,\infty]$, is then
\[
\Bc^p_K(\R^m)\coloneqq \Set{f\in \Ec_K(\C^m)\colon f_0\in L^p(\R^m)},
\]
endowed with the norm
\[
\norm{f}_{\Bc^p_K}\coloneqq \norm{f_0}_{L^p}
\]
for every $f\in \Bc^p_K$.  We remark explicitly that 
\[
\abs{z}_{[-1,1]^m}\meg\sum_{j=1}^m \abs{z_j}
\]
for every $z\in \C^m$, with equality when $z\in \alpha \R^m$ for some
$\alpha\in \C$. Even though the resulting space
$\Ec_{[-\kappa,\kappa]^m}(\C^m)$ is therefore \emph{strictly smaller}
than the space of entire functions of exponential type at most
$\kappa>0$ (according to the above definition,
cf.~\cite{PlancherelPolya}) when $m>1$, the resulting Bernstein spaces
are the same, thanks to Theorem~\ref{prop:7}.

The Paley--Wiener space $\Bc^2_K$ then consists precisely of the holomorphic extensions of the Fourier transforms of the elements of $L^2(K)$, with equality of norms (up to a multiplicative constant). More precisely, (cf., e.g.,~\cite[Chapter III.4]{SteinWeiss}) 
\begin{itemize}
\item[{\tiny $\bullet$}] if $f\in \Bc^2_K(\R^m)$, then $\supp(\widehat{f_0})\subseteq  K$ and $\|f\|_{\Bc^2_K} =\frac{1}{(2\pi)^{m/2}} \|\widehat{f_0}\|_{L^2}$;
\item[{\tiny $\bullet$}] if $g\in L^2(K)$, the function $f$ on $\C^m$ defined by
  \[
  f(z) =\frac{1}{(2\pi)^m} \int_{K} \ee^{iz\xi} g(\xi)\, \dd\xi,
  \]
  for every $z\in \C^m$,  belongs to $\Bc^2_K(\R^m)$ and satisfies $\|f\|_{\Bc^2_K} =\frac{1}{(2\pi)^{m/2}}    \|g\|_{L^2}$. 
\end{itemize}
Analogues of the Paley--Wiener theorems for $\Bc^p_K(\R^m)$, $p\in [1,\infty]$, and of the Plancherel--P\'olya inequalities, $p\in (0,\infty]$, may be considered as \emph{folklore}. 
Explicitly, given $p\in [1,\infty]$, a function $g\in L^p(\R^m)$ is the restriction to $\R^m$ of an element of $\Bc^p_K(\C^m)$ if and only if its Fourier transform $\widehat g$ is supported in $K$. In addition,  for every $p\in(0,\infty]$ and for every $f\in\Bc^p_K(\R^m)$,
\[
\int_{\R^m} |f_y(x)|^p \, \dd x \le \ee^{p|y|_{K^\circ}}\|f\|_{\Bc^p_K}^p
\]
for all $y\in \R^m$, with the obvious modification if $p=\infty$, and where $f_y\colon x \mapsto f(x+i y)$ for every $y\in \R^m$.

In connection with  the Bernstein inequalities, we mention also the so-called \emph{real} Paley--Wiener theorems, where the support of the Fourier transform of a tempered distribution  $u$ is characterized by means the derivatives of  $u$. Cf.~\cite{AndersenDejeu}, also for an extensive review of the literature.
In particular, for $p\in [1,\infty]$ and $f\in L^p(\R^m)$ with $\supp(\widehat f)$  bounded, one may show that
\[
\lim_{k\to \infty} \norm{P(-i \partial)^k f}^{1/k}_{L^p(\R^m)}= \max_{\supp(\widehat f)} \abs{P} 
\]
for every polynomial $P$ on $\R^m$ (cf.~\cite[Theorem 2.9]{AndersenDejeu}). This is sufficient, as $P$ varies, to characterize $\supp(\widehat f)$.

Let us now pass to sampling sequences. When $K$ is a hypercube (hence, by a linear transformation, a non-degenerate parallelotope) centred at the origin, the one-dimensional theory may be extended arguing by induction, whence a precise determination of the $\kappa>0$ for which $\kappa \Z^m$ is a sampling set for $\Bc_K^p$ (cf.~\cite[No.\ 47, Theorem III]{PlancherelPolya}), and the natural analogue of the Whittaker--Kotelnikov--Shannon theorem. If, otherwise, $K$ is an arbitrary symmetric body, more precise sufficient conditions should take into account the shape of $K$ (more precisely, of its polar), cf.~\cite{OlevskiiUlanovskii}.

When $K$ is an arbitrary symmetric body, necessary conditions for a sequence in $\R^m$ to be sampling for $\Bc^2_K$ were determined by Landau~\cite{Landau} in terms of the Beurling lower densities. Actually, Landau obtained necessary conditions
for a sequence in $\R^m$ to be sampling  for the space of $f\in L^2(\R^m)$ with 
$\supp(\widehat{f})\subseteq S$ in terms of the Beurling lower density, where $S\subseteq \R^m$ is any compact set.  This result has been later generalized in several contexts, cf.~\cite{Romeroetal} for a very general formulation and a review of the literature.

\section{Statement of the Main Results}

In this paper we extend the definition of Bernstein spaces and study
their properties along the lines described above.
We  fix a complex hilbertian space $E$ of dimension $n$, a real
hilbertian space $F$ of dimension $m$, and a hermitian map $\Phi\colon
E\times E\to F_\C$.  
We also write $\Phi(\zeta)$ instead of $\Phi(\zeta,\zeta)$,
$\zeta\in E$, to simplify the notation.  We consider the associated
{\em quadratic CR submanifold} $\cM$ of $E\times  F_\C$ given by
\begin{equation}\label{cM-definition}
\cM =  \Set{ (\zeta,z)\in E\times  F_\C\colon \Im z=\Phi(\zeta)}.
\end{equation}
Then, $\cM$ is a CR manifold of CR-dimension $n$ and real codimension
$m$. Cf., e.g.,~\cite{Boggess} for an extensive treatment of CR
manifolds, and also~\cite{PR,Calzi3} for quadratic (or quadric) CR
manifolds.  Analysis and geometry on quadratic CR manifolds is
motivated by the fact that these manifolds constitute local models for
all CR manifolds. Besides that, quadratic CR manifold have attracted
the interest of several authors in their own right, cf.,
e.g.,~\cite{Forstneric,Boggess2,PR,BR,BR1,BR2,BR3,Calzi3,Calzi4},   and
references therein.  

The manifold $\cM$ can be canonically identified with a $2$-step
nilpotent Lie group $\Nc$ as follows. 

\begin{deff}\label{Nc:def}
We define $\Nc\coloneqq E\times F$, endowed with the $2$-step nilpotent Lie group structure given by the product
\[
(\zeta,x)(\zeta',x')\coloneqq (\zeta+\zeta', x+x'+2 \Im \Phi(\zeta,\zeta'))
\]
for $(\zeta,x),(\zeta',x')\in E\times F$.
\end{deff}

It turns out that $\Nc$ acts freely and affinely on the complex space $E\times F_\C$, by
\[
(\zeta,x)\cdot (\zeta',z')\coloneqq (\zeta+\zeta',x+i\Phi(\zeta)+z'+2 i\Phi(\zeta',\zeta))
\]
In particular, $\Nc$ acts simply transitively on the CR submanifold $\cM=\Nc\cdot \vect{0}$, with which it may therefore be identified. We shall also say that $\Nc$, with the CR structure induced by $\cM$, is a quadratic CR manifold. 

We note that the Lebesgue measure on $\Nc=E\times F$ is a bi-invariant Haar measure; we shall denote it by $\dd(\zeta,x)$ in integrals. 
For every $p\in(0,\infty]$, we denote by $L^p(\Nc)$ the standard Lebesgue space
with respect to this measure.  We also denote by $Q\coloneqq2n+2m$ the {\em  homogeneous dimension} of $\Nc$  with respect to the dilations $t \cdot (\zeta,x)\coloneqq (t \zeta,t^2 x)$. 

We now introduce Siegel CR submanifolds, which are those on which non-trivial Bernstein spaces of type $L^p$, $p\in (0,\infty)$, may be defined (cf.~\cite[Corollary 5.9]{Calzi3} and  Theorem~\ref{cor:1}). 

\begin{deff}\label{Siegel-manifold-deff}
Define
\[
\Lambda_+\coloneqq \Set{\lambda\in F'\colon \forall \zeta\in
  E\setminus \Set{0}\:\: \la \lambda , \Phi(\zeta) \ra>0}.
\] 
We say that $\cM$ is a \emph{Siegel CR submanifold} of $E\times F_\C$ is $\Lambda_+\neq \emptyset$. In this case, we say that $\Nc$ is a \emph{Siegel CR manifold}.
\end{deff}

We recall that, if $A$ is a subset of $F$, then its polar is
\[
A^\circ\coloneqq \Set{\lambda\in F'\colon \forall x\in A\:\:\langle \lambda, x\rangle\Meg -1},
\]
so that $A^\circ$ is a closed convex subset of $F'$ containing $0$ (cf.~\cite[Definition 2 of Chapter II, \S\ 6, No.\ 3]{BourbakiTVS}). The polar of a subset of $F'$ is defined analogously (as a subset of $F$), so that $A^{\circ\circ}$ is the closed convex envelope of $A\cup \Set{0}$ by the bipolar theorem (cf.~\cite[Theorem 1 of Chapter II, \S\ 6, No.\ 3]{BourbakiTVS}). 

\begin{oss}\label{CR-remark}
Notice that $\Lambda_+$ is an open convex cone in $F'$. In addition, if $\Lambda_+\neq \emptyset$,  then  $\Phi$ is non-denegerate and $\Lambda_+$ is the interior of the polar of $\Phi(E)$, so that $\Phi(\zeta)\in \Lambda_+^\circ$ for every $\zeta\in E$, (cf., e.g.,~\cite[Proposition 2.5]{Calzi3}). 
Furthermore, $\Phi$ is $\overline\Omega$-positive for every open convex cone $\Omega$ not containing affine lines and
such that $\Omega'\subseteq \Lambda_+$, where $\Omega'$ denotes the
dual of $\Omega$, that is, the interior of its polar. In particular,
for every such $\Omega$, 
\[
D_\Omega\coloneqq \Set{(\zeta,z)\in E\times F_\C\colon \Im z-\Phi(\zeta)\in \Omega}
\]
is a Siegel domain whose \v Silov boundary is $\cM$. 
\end{oss}

\begin{oss}\label{foliation-remark}
Notice that the space $E\times F_\C$ can be foliated by copies of $\cM$, hence of $\Nc$, since
\[
(\zeta,z)= (\zeta, x+i\Phi(\zeta)+ih)
\]
where $h=\Im z-\Phi(\zeta) \eqqcolon\rho(\zeta,z)\in F$.  Then,
\[
E\times F_\C = \bigcup_{h\in F} (\cM+(0,ih) )=  \bigcup_{h\in F}
\Set{ (\zeta,z+ih)\colon (\zeta,z)\in\cM, h\in F}.
\]
Given a function $f$ defined on $E\times F_\C$, we denote by $f_h$ its restriction to
$\cM+ih$, interpreted as a function on $\Nc$. In other words, for $(\zeta,x)\in\Nc$ and $h\in F$, we write
\[
f_h(\zeta,x) = f(\zeta, x+i\Phi(\zeta)+ih).
\]
\end{oss}

For some general facts on quadratic CR manifolds see~\cite{Boggess,PR}, and for a recent account of analysis on Siegel Type II domain see~\cite{CalziPeloso,CP1,CP2}. Fourier analysis on $\Nc$ plays a relevant role in the study of various function spaces of holomorphic functions on $D$, and in this work.  For more on this, see also~\cite{Calzi,Calzi2}. \medskip

Given a compact convex subset $K$ of $F'$ (not  necessarily {\em symmetric}), we  denote by $H_K $ the supporting  function of $K$: 
\begin{equation}\label{supp-func:def}
H_K (h) =\sup_{\lambda\in -K} \la\lambda, h\ra \in [-\infty,\infty),
\end{equation}
see, e.g.,~\cite[Exercise 9 of Chapter II, \S\ 6]{BourbakiTVS} or~\cite[Section 4.3]{Hormander} (where a slightly different convention is adopted). In particular, $H_\emptyset(h)=-\infty$ and $H_K(h)>-\infty$ for every every non-empty compact convex subset $K$ of $F'$, and for every $h\in F$.  The trivial case $K=\emptyset$ is considered for notational convenience.
Further, $H_K$ is convex, continuous, and sub-additive.
Observe that, if $0\in K$, then $H_K$ is the Minkowski functional (or gauge) associated with $K^\circ$. In particular,  if $K$ is a symmetric body, then $H_K(h)= |h|_{K^\circ} $ for every $h\in F$.
We observe explicitly that, if $K_1,K_2$ are two compact convex subsets of $F'$, then
\[
H_{K_1+K_2}=H_{K_1}+H_{K_2}.
\]
In particular,
\begin{equation}\label{eq:1}
H_{\lambda+K}(h)=H_K(h)-\langle \lambda,h\rangle\qquad \text{and}\qquad H_{K+\overline B_{F'}(0,r)}(h)= H_K(h)+r \abs{h}
\end{equation}
for every $h\in F$, for every $\lambda\in F'$, and for every $r>0$.

On $E\times F_\C$ we introduce a modified space of entire functions of
exponential type $K$.

\begin{deff}\label{exp-type-Kirc-deff}
Let $K\subseteq F'$ be a compact convex  set and let $H_K$ be as in~\eqref{supp-func:def}.  We define a modified space $\widetilde \Ec_K (E\times F_\C,\Nc)$ of entire functions of exponential type $K$ as 
\[
\Set{f\in\Hol(E\times F_\C)\colon \forall \eps>0\: \exists C_\eps>0\: \forall (\zeta,z)\in E\times F_\C \:\: \abs{f(\zeta,z)}\meg C_\eps \ee^{C_\eps(\abs{\zeta}^2+\abs{\Re z})+H_{K+\overline B_{F'}(0,\eps)}(\rho(\zeta,z))} }.
\] 
\end{deff}

Observe that, if $K$ is a symmetric body and $E=\Set{0}$, then $\Ec_K(F_\C)$ is a \emph{proper} subset of $\widetilde\Ec_K(F_\C, F)$, since the elements of the latter space are allowed to grow faster in the directions of $F$.  
Notice that, since $K$ is now allowed to be non-symmetric, $H_K$ is no longer a norm, so that it may not be `canonically' extended to $F_\C$. Further, when $E\neq \Set{0}$, the dependence of $h=\rho(\zeta,z)$ on $(\zeta,z)$ is quadratic in $\zeta$, so that we chose to allow a quadratic exponential growth in the directions of $E$. As the proof of Theorem~\ref{prop:6} shows (cf.~Remark~\ref{oss:1}), though, a much faster growth may be allowed without changing the definition of the corresponding Bernstein spaces.
The reason why non-symmetric $K$ are considered lies in  Theorem~\ref{cor:1}.

We now define the Bernstein spaces we are interested in.
\begin{deff}
Take $p\in (0,\infty]$ and a compact convex subset $K$ of
$F'$. We define the  Bernstein space 
\[
\Bc^p_K(E\times F_\C,\Nc)\coloneqq \Set{f\in \widetilde\Ec_K(E\times F_\C,\Nc)\colon f_0\in  L^p(\Nc)  },
\]
endowed with the norm
\[
\norm{f}_{\Bc^p_K}\coloneqq \norm{f_0}_{L^p(\Nc)} 
\]
for every $f\in \Bc^p_K(E\times F_\C,\Nc)$. We shall simply write $\Bc^p_K(\Nc)$ or $\Bc^p_K$ if no ambiguity may arise.
\end{deff}

The above definition extends the one introduced in~\cite{MonguzziPelosoSalvatori} in the case in which $\Nc$ coincides with the Heisenberg group, that is, $n\ge1$ and $m=1$. 
Notice that, when $E=\Set{0}$,  multiplication by $\ee^{i\langle
  \lambda_\C,\,\cdot\,\rangle}$ induces an isometric isomorphism of
$\Bc^p_K(F)$ onto $\Bc^p_{\lambda+K}(F)$, so that the classical theory
trivially allows to deal with the case in which $K$ is a compact
convex set admitting a centre of symmetry. Notice, further, that
multiplication by the function $(\zeta,z)\mapsto\ee^{i\langle
  \lambda_\C,z\rangle}$ induces a bounded operator from $\Bc^p_K(\Nc)$
into $\Bc^p_{\lambda+K}(\Nc)$ if and only if
$\lambda\in \Phi(E)^\circ$ (cf.~Proposition~\ref{prop:11}). 
We can now state our main results.

\begin{teo}[Plancherel--P\'olya Inequality]\label{prop:6}
Let $K$ be a compact convex subset of $F'$ and take $p\in (0,\infty]$. Then, 
\[
\norm{f_h}_{L^p(\Nc)}\meg  \ee^{H_{K}(h)} \norm{f_0}_{L^p(\Nc)}
\]
for every $f\in  \Bc^p_{K }(\Nc)$  and for every $h\in F$.
\end{teo}

As a consequence of the preceding result and~\cite[Proposition 5.8 and Corollaries 5.9 and 5.10]{CalziPeloso}, we may now indicate Paley--Wiener theorems for the spaces $\Bc^p_K$ and draw some consequences. 
We need to introduce some more notation.

For every $w\in E$ we denote by $Z_w$ the the left-invariant vector field on $\Nc$ that coincides with the Wirtinger derivative $\frac12 (\partial_w -i\partial_{i w})$ at the origin. 
Explicitly, $Z_w=\frac12 (\partial_w -i\partial_{i w})+ i\Phi(w,\,\cdot\,)\partial_F$, that is, 
\begin{align*}
Z_w f(\zeta,x) &= \frac12 (\partial_w -i\partial_{i w})f(\zeta,x) +i(\partial_{\Re \Phi(w,\zeta)}f)(\zeta,x) - (\partial_{\Im \Phi(w,\zeta)}f)(\zeta,x)
\end{align*}
for any $f\in C^1(\Nc)$ and every $(\zeta,x)\in\Nc$.  

We recall the definition of CR functions and distributions on $\Nc$.

\begin{deff}
Let $f$ be a function of class $C^1$ (or, more generally, a distribution) on $\Nc$. Then, $f$ is CR if $\overline{Z_w} f=0$ for every $w\in E$.
\end{deff}

For technical reasons, it will be convenient to consider the following space $\Oc_K(\Nc)$ of smooth functions on $\Nc$. We denote by $\Fc_F$ the (commutative) Fourier transform on $F$. 

\begin{deff}\label{OK:def}
Given a compact subset  $K$ of $F'$, define
\[
\Oc_K(\Nc) \coloneqq\Set{f\in C^\infty(\Nc)\cap \Sc'(\Nc)\colon \text{$f$ is CR, }  \forall \zeta\in E\:\:\Fc_F[f(\zeta,\,\cdot\,)] \text{ is supported in } K}.
\]
\end{deff}
For more on this space, see~\cite{Calzi3}. 

\begin{teo}\label{cor:1}
Let $K$ be a compact convex subset of $F'$ and take $p\in (0,\infty]$. Then, the following hold:
\begin{enumerate}
\item[\textnormal{(1)}] the mapping $f \mapsto f_0$ induces an isomorphism of $\Bc^p_K(\Nc)$ onto $\Oc_K(\Nc)\cap L^p(\Nc)$ (endowed with the topology induced by $L^p(\Nc)$);

\item[\textnormal{(2)}] if $p<\infty$, then $\Bc^p_K(\Nc)=\Bc^p_{K\cap \overline{\Lambda_+}}(\Nc)$, while $\Bc^\infty_K(\Nc)=\Bc^\infty_{K\cap \Phi(E)^\circ}(\Nc)$;

\item[\textnormal{(3)}] there are a complex subspace $E_1$ of $E$ and a real subspace $F_1$ of $F$ such that $ F_1^\circ$ is the vector space generated by $K\cap \Phi(E)^\circ$, $\Nc_1=E_1\times F_1$ is a normal subgroup of $\Nc$,  $\Nc/\Nc_1$ is a \emph{Siegel} CR manifold, and the mapping
\[
\Bc^\infty_K(\Nc/\Nc_1)\ni f \mapsto f \circ \pi\in \Bc^\infty_K(\Nc)
\]
is an isomorphism, where $\pi\colon E\times F_\C\to E\times F_\C/(E_1\times F_1)_\C $ is the canonical projection. 
\end{enumerate}
\end{teo}

In particular, the spaces $\Bc^p_K(\Nc)$, for $p<\infty$, are
non-trivial if and only if $K\cap \Lambda_+$ has a non-empty interior,
while the study of the space $\Bc^\infty_K(\Nc)$ may be essentially
reduced to the case in which $K\cap \Lambda_+$ has a non-empty
interior.

Therefore, from now on, 
we shall always assume that $\Nc$ is a \emph{Siegel} CR manifold, that
is, that $\Lambda_+\neq \emptyset$.  Since our results hold (possibly with slightly different proofs) also in the general case, we shall not recall this assumption in the statements.   We point
out that Siegel CR manifolds appear in a very natural way in the
analysis of the regularity of the Kohn Laplacian and they are precisely the quadratic CR manifolds
whose Kohn Laplacian is non-solvable on $0$-forms, or equivalently
that its $L^2$-null-space is not trivial, see~\cite[Theorem 1]{PR}.

We remark explicitly that, if $E=\Set{0}$, then $\Nc=F$ is abelian, $\Lambda_+=F'$, and $\Bc^p_K(\Nc)$ reduces to the classical Bernstein space considered in the Introduction. 

We now consider some examples, in which one may easily infer when
the spaces $\Bc^p_{ K}(\Nc)$ are non-trivial.
\begin{example}
	It is proved in~\cite[Theorem 2 of Section 7.3]{Boggess} that, if $E=\C^2$ and $F=\R^2$, then, up to a linear change of coordinates, the following cases may occur:
	\begin{itemize}
\item[{\tiny $\bullet$}] there is a hermitian form $\phi$ on $E$ such that $\Phi(\zeta)=(\phi(\zeta),0)$ for every $\zeta\in \C^2$. In this case, $\Phi(E)$ is either $\Set{(0,0)}$, $\R_+\times \Set{0}$, $\R_-\times \Set{0}$, or $\R\times \Set{0}$ according as $\phi$ is $0$, positive, negative, or indefinite, so that $\Phi(E)^\circ$ is $\R^2$, $\R_+\times \R$, $\R_-\times \R$, or $\Set{0}\times \R$, respectively, and $\Lambda_+$ is its interior, if $\phi$ is non-degenerate, the empty set otherwise;
		
\item[{\tiny $\bullet$}] $\Phi(\zeta)=(\abs{\zeta_1}^2, \abs{\zeta_2}^2)$ for every $\zeta=(\zeta_1,\zeta_2)\in \C^2$. In this case, $\Phi(E)=\R_+^2$ and $\Lambda_+=(\R_+^*)^2$;
		
\item[{\tiny $\bullet$}] $\Phi(\zeta)=(\abs{\zeta_1}^2, \Re(\zeta_1\overline{\zeta_2}))$ for every $\zeta=(\zeta_1,\zeta_2)\in \C^2$. In this case, $\Phi(E)=\Set{0}\cup(\R_+^*\times \R)$ and $\Lambda_+=\emptyset$;
		
\item[{\tiny $\bullet$}] $\Phi(\zeta)=(\Re(\zeta_1\overline{\zeta_2}),\Im(\zeta_1\overline{\zeta_2})))$ for every $\zeta=(\zeta_1,\zeta_2)\in \C^2$. In this case, $\Phi(E)=\R^2$ and $\Lambda_+=\emptyset$.
	\end{itemize}
\end{example}

\begin{example}
	Take $r,k\in \N$, and define:
	\begin{itemize}
		\item[{\tiny $\bullet$}] $E$ as the space of $k\times r$ over $\C$;
		
		\item[{\tiny $\bullet$}] $F$ as the space of self-adjoint $r\times r$ matrices over $\C$;
		
		\item[{\tiny $\bullet$}] $\Phi(\zeta)\coloneqq \zeta\zeta^*$ for every $\zeta\in E$.
	\end{itemize}
	Then, the closed convex envelope of $\Phi(E)$ is the cone of positive self-adjoint $r\times r$ matrices over $\C$ (cf.~\cite[Example 2.14 and Corollary 2.58]{CalziPeloso}), so that $\Lambda_+$ may be identified with the cone of non-degenerate self-adjoint $r\times r$ matrices over $\C$ by means of the bilinear form $F\times F\ni (x,y)\mapsto \tr(x y)$ (cf.~\cite[Example 2.6]{CalziPeloso}). If $k\Meg r$, then $\Phi(E)$ is already convex and closed.
\end{example}

We now pass to an extension of  the Bernstein inequality and a corresponding characterization of $\Bc_K^p(\Nc)$.  Notice that in the next results we only consider derivatives in the directions in $\Nc$ that are transversal to $E$.

\begin{teo}[Bernstein Inequality]\label{Bern-ineq-thm}
Let $K$ be a \emph{symmetric} compact convex subset of $F'$, and take $p\in [1,\infty]$. 
Then, for every $f\in \Bc_K^p(\Nc)$, for every $v_1,\dots, v_k\in F$,
\[
\norm{\partial_{v_1}\cdots \partial_{v_k} f_0}_{L^p(\Nc)}\meg \bigg(
\prod_{j=1}^k H_{K}(v_k)\bigg)  \norm{f_0}_{L^p(\Nc)}. 
\]
Conversely,  take a CR $\phi\in L^p(\Nc)$ such that $\partial_v^k \phi \in L^p(\Nc)$ for every $v\in \partial B_F(0,1)$ and for every $k\in \N$, and such that 
\[
\limsup_{k\to \infty}\norm{\partial_v^k \phi}_{L^p(\Nc)}^{1/k}\meg H_{K}(v)
\]
for every $v\in \partial B_F(0,1)$.
Then, there is a unique $f\in \Bc_{K}^p(\Nc)$ such that $f_0=\phi$.
\end{teo}

We will also prove a weaker result in the case of derivatives in the
complex tangential directions of $\Nc$, see Proposition~\ref{comp-tang-dir:Prop}. 
We shall now describe a `real Paley--Wiener theorem', which extends~\cite[Theorem 2.9]{AndersenDejeu}. 

\begin{deff}
Take $p\in (0,\infty]$. We say that $f\in L^p(\Nc)$ has  bounded  spectrum if $f\in \Oc_K(\Nc)$ for some compact subset $K$ of $F'$. We shall then define
\[
\Sp(f)\coloneqq \bigcap\Set{K\colon \text{$K$ is a compact subset of $F$, } f\in \Oc_K(\Nc)}=\overline{\bigcup_{\zeta\in E} \supp(\Fc_F[f(\zeta,\,\cdot\,)])}
\]
\end{deff}

The following result is based on~\cite{Bang,AndersenDejeu}, and may be considered as a real Paley--Wiener theorem.

\begin{prop}\label{prop:9}
Take $p\in (0,\infty]$, $f\in L^p(\Nc)$ with bounded spectrum, and a polynomial $P$ on $F'$. Then,
\[
\lim_{k\to \infty} \norm{P(-i \partial_F)^k f}_{L^p(\Nc)}^{1/k}= \max_{\Sp(f)} \abs{P}.
\]
\end{prop}

Here, $P(-i \partial_F) $ is convolution by $\Fc_F^{-1}(P)$, interpreted as a distribution on $\Nc$. 
Notice that, as observed in~\cite[Lemma 2.8]{AndersenDejeu}, 
\[
\Sp(f)= \Set{\lambda\in F'\colon \forall P \:\: \abs{P(\lambda)}\meg \max_{\Sp(f)}\abs{P}},
\]
where $P$ runs through the set of polynomials on $F'$. Thus, the preceding result fully characterizes $\Sp(f)$ if $f$ has bounded spectrum.

Next we turn to the question of describing sampling sequences for $\Bc_K^p(\Nc)$.  
We endow $\Nc$ with a left-invariant distance which is homogeneous of degree $1$ with respect to the dilations given by $t\cdot (\zeta,x)=(t\zeta, t^2 x)$.\footnote{Distances of this kind always exist and may be defined, e.g., as suitable Carnot--Carathéodory distances. One may also assume that they are of class $C^\infty$ on the complement of the diagonal, cf.~\cite{HebischSikora}. }
We recall that, given a metric space $(X,d)$, $\delta,R>0$, a
sequence  $(x_k)$ is a $(\delta,R)$-{\em lattice} if the balls $B(x_k,\delta)$ are disjoint, while the balls $B(x_k,R\delta)$ form a cover of $X$.  

\begin{teo}\label{prop:10}
Take $p\in(0,\infty]$, $R_0>0$, $\delta_+>0$, and a  compact convex subset $K$ of $F'$.  Then, there are two constants $C,\delta_->0$ such that for every $(\delta,R)$-lattice $(\zeta_j,x_j)_{j\in J}$ on $\Nc$, with $\delta>0$ and $R\in (1,R_0]$, one has 
\[
\frac1 C \norm{f}_{L^p(\Nc)}\meg \delta^{Q/p}\norm*{\max_{\overline B((\zeta_j,x_j), R\delta)} \abs{f}  }_{\ell^p(J)}\meg  C \norm{f}_{L^p(\Nc)}
\]
for every $f\in \Oc_K(\Nc)$, if $\delta\meg \delta_+$, and 
\[
\frac1 C \norm{f}_{L^p(\Nc)}\meg \delta^{Q/p}\norm*{\min_{\overline B((\zeta_j,x_j), R\delta)} \abs{f}  }_{\ell^p(J)}\meg  C \norm{f}_{L^p(\Nc)}
\]
for every $f\in \Oc_K(\Nc)$, if $\delta\meg \delta_-$.
\end{teo}

Notice that the preceding result allows to characterize the elements of $\Oc_K(\Nc)$ which belong to $L^p(\Nc)$ on the sole knowledge of their samples on suitable countable sets.

\medskip

The paper is organized as follows.   In Section~\ref{sec:2}, we
describe some basic facts on the Fourier transform of CR function on
$\Nc$ and on Paley--Wiener--Schwartz theorems associated with $\Nc$. 
In Section~\ref{P-L:sec}, we prove Theorem~\ref{prop:6} and some of
its consequences, such as the completeness of the spaces $\Bc^p_K$
(cf.~Corollary~\ref{cor:2}), the inclusions $\Bc^p_K\subseteq \Bc^q_K$
for $p\meg q$ (cf.~Theorem~\ref{prop:7}), and a characterization of
the $K$ for which $\Bc^p_K$ is not trivial
(cf.~Proposition~\ref{prop:3}). 
In Section~\ref{sec:4}, we prove Theorem~\ref{Bern-ineq-thm} and
Proposition~\ref{prop:9}, as well as a version of Bernstein
inequalities for derivatives of the form $Z_v$, $v\in E$
(cf.~Proposition~\ref{comp-tang-dir:Prop}). 
In Section~\ref{sec:5}, we shall determine the reproducing kernel of
$\Bc^2_K$ (cf.~Proposition~\ref{prop:12}) and study the boundedness of
the corresponding projector on $L^p(\Nc)$. In particular, as a
consequence of Fefferman's solution of the multiplier problem for the
ball~\cite{Fefferman}, we shall prove that, if $K$ is contained in
$\Lambda_+$ and is sufficiently `far' from being polyhedral, then $P$
is not bounded on $L^p$ for any $p\neq 2$
(cf.~Corollary~\ref{cor:3}). 
In Section~\ref{sec:6} we shall prove Theorem~\ref{prop:10} and show
how lattices on $\Nc$ may be constructed from lattices on $E$ and $F$
(cf.~Lemma~\ref{lem:2}).

\section{Analysis on Siegel CR Manifolds}\label{sec:2}

\subsection{Fourier Analysis on $\Nc$}

For a more extensive description of Fourier analysis on $\Nc$, we refer the reader to~\cite{Calzi3}. Here we shall limit ourselves to a description of the Plancherel formula for CR elements of $L^2(\Nc)$.

Observe that, for every $\lambda\in \Lambda_+$, the quotient $\Nc/\ker \lambda$ is isomorphic to a Heisenberg group (to $\R$ if $E=\Set{0}$), so that the Stone--von Neumann theorem (cf., e.g.,~\cite[Theorem 1.50]{Folland}) implies that there is (up to unitary equivalence)  a unique irreducible continuous unitary representation $\pi_\lambda$ of $\Nc$ in some hilbertian space $\sH_\lambda$ such that $\pi_\lambda(0,x)=\ee^{-i \langle \lambda, x \rangle}$ for every $x\in F$.
One may then choose $\sH_\lambda=\Hol(E)\cap L^2(\nu_\lambda)$, where $\dd \nu_\lambda(\omega)= \ee^{-2 \langle \lambda, \Phi(\omega)\rangle}\dd \omega$, and
\[
\pi_\lambda(\zeta,x) \psi(\omega)= \ee^{ \langle \lambda_\C, -i x+2 \Phi(\omega,\zeta)-\Phi(\zeta)\rangle} \psi(\omega-\zeta)
\] 
for every $(\zeta,x)\in \Nc$, for every $\omega\in E$, and for every $\psi\in \sH_\lambda$. 
In addition, for every $\lambda\in \Lambda_+$, we define $\abs{\Pfaff(\lambda)}$ as the determinant of the \emph{positive} hermitian form $\langle \lambda,\Phi\rangle$ with respect to the scalar product of $E$. Observe that $\abs{\Pfaff}$ is  positive on $\Lambda_+$ and extends to a polynomial on $F'$.

The following result summarizes~\cite[Propositions 2.3, 2.4, 2.5, and 2.6]{Calzi3}.
We denote by $\Lin(\sH_\lambda)$ the space of endormophisms of $\sH_\lambda$, and by $\Lin^2(\sH_\lambda)$ the space of Hilbert--Schmidt endomorphisms of $\sH_\lambda$.

\begin{prop}\label{prop:4}
The following hold:
\begin{enumerate}
\item[\textnormal{(1)}] $\pi_\lambda$ is an irreducible continuous unitary representation of $\Nc$ in $\sH_\lambda$ for every $\lambda\in \Lambda_+$;

\item[\textnormal{(2)}] for every $\lambda\in \Lambda_+$, $e_{\lambda,0}\coloneqq \sqrt{\frac{2^n\abs{\Pfaff(\lambda)}}{\pi^n}}$ is a unit vector in $\sH_\lambda$, and
\[
\langle \pi_\lambda(\zeta,x)e_{\lambda,0}\vert e_{\lambda,0}\rangle=\ee^{-\langle \lambda_\C , i x+\Phi(\zeta)\rangle}
\]  
for every $(\zeta,x)\in \Nc$;

\item[\textnormal{(3)}] for every $\lambda\in \Lambda_+$ and for every $v\in E$, 
\[
\dd \pi_\lambda(Z_v)=-\partial_v \qquad \text{and} \qquad \dd \pi_\lambda(\overline{Z_v})=2 \langle \lambda, \Phi(\,\cdot\,,v)\rangle;
\] 

\item[\textnormal{(4)}] the mapping 
\[
L^1(\Nc)\ni f \mapsto (\pi_\lambda(f))\in \prod_{\lambda\in \Lambda_+} \Lin(\sH_\lambda)
\]
induces an isomorphism
\[
L^2(\Nc) \to\frac{2^{n-m}}{\pi^{n+m}} \int_{\Lambda_+}^\oplus \Lin^2(\sH_\lambda) P_{\lambda,0} \abs{\Pfaff(\lambda)}\,\dd \lambda,
\]
where $P_{\lambda,0}=\langle \,\cdot\,\vert e_{\lambda,0}\rangle e_{\lambda,0}$ for every $\lambda\in\Lambda_+$.
\end{enumerate}
\end{prop}

\begin{deff}
For every $\lambda\in \Lambda_+$, we define $e_{\lambda,0}$ and $P_{\lambda,0}$ as in Proposition~\ref{prop:4}.
\end{deff}

The following is a consequence of~\cite[Proposition 2.7]{Calzi3}.

\begin{prop}\label{prop:5}
Take a CR element $f$ of $L^2(\Nc)$ and a closed subset $K$ of $\overline{\Lambda_+}$. Then, the following conditions are equivalent:
\begin{enumerate}
\item[\textnormal{(1)}] $\pi_\lambda(f)=0$ for almost every $\lambda\in \Lambda_+\setminus K$;

\item[\textnormal{(2)}] $\Fc_F[f(\zeta,\,\cdot\,)]$ is supported in $K$ for almost every $\zeta\in E$.
\end{enumerate}
\end{prop}

\subsection{Paley--Wiener--Schwartz Theorems}

We now indicate the significance of the spaces $\Oc_K(\Nc)$ defined in the Introduction in relation with the Paley--Wiener--Schwartz theorems on $\Nc$.

The following results summarizes~\cite[Theorem 3.3 and Proposition 5.7]{Calzi3}.

\begin{prop}\label{prop:1}
Let $K$ be a  compact convex subset of $F'$. Then, the following hold:
\begin{enumerate}
\item[\textnormal{(1)}] if $f\in \Hol(E\times F_\C)$ and there are $N,C>0$ such that
\begin{equation}\label{eq:2}
\abs{f(\zeta,z)}\meg C (1+\abs{\zeta}^2+\abs{z})^N \ee^{H_K(\rho(\zeta,z))}
\end{equation}
for every $(\zeta,z)\in E\times F_\C$, then $f_0\in \Oc_{K\cap \Phi(E)^\circ}(\Nc)$;

\item[\textnormal{(2)}] if $u\in \Oc_K(\Nc)$, then there is a unique $f\in \Hol(E\times F_\C)$ such that $f_0=u$; in addition,
\[
f(\zeta,z)=\frac{1}{(2\pi)^m} \int_K \Fc_F[u(\zeta,\,\cdot\,)](\lambda )\ee^{i \langle \lambda_\C, z\rangle}\,\dd \lambda
\]
for every $(\zeta,z)\in E\times F_\C$, and $f$ satisfies the estimates~\eqref{eq:2}.
\end{enumerate}
\end{prop}

We shall now define a subspace of $\Sc(\Nc)\cap \Oc_K(\Nc)$ which is particularly simply described in terms of the Fourier transform. 
Observe that, for $\Sc(\Nc)\cap \Oc_K(\Nc)$ to be non-trivial, $K\cap \overline{\Lambda_+}$ must have a non-empty interior, because of Proposition~\ref{prop:4}. For this reason, we shall consider only closed subsets of $\overline{\Lambda_+}$.

\begin{deff}
Let $K$ be a closed subset of $\overline{\Lambda_+}$. Then, we define
\[
\widetilde \Sc_K(\Nc)\coloneqq \Set{\phi\in \Sc(\Nc)\colon \text{$\phi$ is CR and }\forall \lambda\in \Lambda_+\:\: \pi_\lambda(\phi)=\chi_K(\lambda)P_{\lambda,0}\pi_\lambda(\phi)P_{\lambda,0} }
\]
and
\[
\Fc_\Nc\colon \widetilde \Sc_{\overline{\Lambda_+}}(\Nc)\ni \phi \mapsto (\tr \pi_\lambda(\phi))_\lambda.
\]
\end{deff}

The following result is a consequence of~\cite[Proposition 5.2]{Calzi3}.

\begin{prop}\label{prop:8}
Let $K$ be a closed subset of $\overline{\Lambda_+}$. Then, $\Fc_\Nc$ induces an isomorphism of $\widetilde \Sc_K(\Nc)$ onto the space of Schwartz functions on $F'$ supported in $K$.
In addition, for every $f\in \widetilde \Sc_K(\Nc)$,
\[
f(\zeta,x)=\frac{2^{n-m}}{\pi^{n+m}}\int_K (\Fc_\Nc f)(\lambda) \abs{\Pfaff(\lambda)} \ee^{\langle\lambda_\C, i x-\Phi(\zeta)\rangle}\,\dd \lambda
\]
for every $(\zeta,z)\in \Nc$, and
\[
\Fc_F[f(\zeta,\,\cdot\,)](\lambda)=\frac{2^n \abs{\Pfaff(\lambda)}}{\pi^n}  \ee^{-\langle \lambda, \Phi(\zeta)\rangle}(\Fc_\Nc f)(\lambda)
\]
for every $\zeta\in E$ and for every $\lambda\in \Lambda_+$, while $\Fc_F[f(\zeta,\,\cdot\,)](\lambda)=0$ for every $\zeta\in E$ and for every $\lambda\in F'\setminus \Lambda_+$.
\end{prop}

\begin{proof}
The first three assertions follow from~\cite[Proposition 5.2 and its proof]{Calzi3}, while the last assertion follow from Proposition~\ref{prop:5}.
\end{proof}

\section{The Plancherel--P\'olya Theorem and its consequences}\label{P-L:sec}

In this section we prove Theorem~\ref{prop:6}. 
The case in which $\Nc$ is abelian and $K$ is a cube centred at the
origin  is essentially~\cite[No.\ 45]{PlancherelPolya}. The general
case follows from the one-dimensional case.  As a consequence, we deduce that the spaces
$\Bc_K^p(\Nc)$ are Banach spaces for $p\in[1,\infty]$ and quasi-Banach
spaces for $p\in(0,1)$.

\begin{proof}[Proof of  Theorem~\ref{prop:6}.]
\textsc{Step I.} We assume first that $\Nc=\R$, so that
$\Lambda_+=\R$. We may assume that $K\neq \emptyset$. Then, take $f\in
\Bc^p_K(\C)$. Observe that, up to replacing $f$ with  $\ee^{-i
  \alpha\,\cdot\,}f$,  where $\alpha$ is the mid-point of $K$, we may
assume that $K=[-\kappa ,\kappa ]$, so that $H_K(x)=\kappa \abs{x}$
for every $x\in \R$.  In addition, if $p<\infty$, then~\cite[No.\
30]{PlancherelPolya} implies that $f_0\in L^\infty(\Nc)$, since $f$ is
of exponential type (\emph{a priori} possibly larger than $\kappa$).  
Therefore,~\cite[Lemma 3.1]{MonguzziPelosoSalvatori} implies that  
\[
\abs{f(z)}\meg \norm{f_0}_{L^\infty(\R)}\ee^{\kappa \abs{\Im z}}
\]
for every $z\in \C$, so that $f$ is of exponential type at most $\kappa$ and the assertion follows in this case if $p=\infty$. If, otherwise, $p<\infty$, the assertion follows in this case thanks to~\cite[No.\ 27]{PlancherelPolya}.

\textsc{Step II.} Now, consider the general case.
Define, for every $h\in \partial B_F(0,1)$, for every $\zeta\in E$, and for every $x\in (\R h)^\perp$,
\[
f^{(\zeta,x,h)}\colon\C \ni w \mapsto f(\zeta,x+i\Phi(\zeta)+w h)\in \C,
\]
so that $f^{(\zeta,x,h)}$ is holomorphic on $\C$, and for every $\eps>0$ there is $C_{\eps,h,\zeta,x}>0$ such that
\[
\abs{f^{(\zeta,x,h)}(w)}\meg C_{\eps,h,\zeta,x}\ee^{ C_{\eps,h,\zeta,x}\abs{\Re w}+ H_{K_\eps}((\Im w )h)}
\]
for every $w\in \C$, where $K_\eps\coloneqq K+\overline B_{F'}(0,\eps)$.
Observe that, if $\pi_h$ is the self-adjoint projector of $F'$ onto
$((\R h)^\circ)^\perp$, canonically identified with the dual of $\R
h$, then 
\[
H_{K_\eps}( t h)= H_{\pi_h(K_{\eps})}(t h)
\]
for every $t\in \R$. Therefore, we may find a closed bounded interval
$K_{\eps,h}$ in $\R$ in such a way that $H_{K_\eps}(t h)=H_{K_{\eps,h}}(t)$ for every $t\in\R$. 
In addition, $f^{(\zeta,x,h)}_0=f_0(\zeta,x+\,\cdot\, h)\in L^p(\R)$
for almost every $(\zeta,x)\in E\times (\R h)^\perp$, so that
$f^{(\zeta,x,h)}\in \Bc_{K_{\eps,h}}^p(\R)$, for almost every $(\zeta,x)\in
E\times (\R h)^\perp$. Then, \textsc{step II} implies that 
\[
\norm{f^{(\zeta,x,h)}_t}_{L^p(\R)}\meg \ee^{H_{K_{\eps,t}}(t)} \norm{f^{(\zeta,x,h)}_0}_{L^p(\R)} 
\]
for every $t\in \R$ and for almost every $(\zeta,x)\in E\times (\R h)^\perp$. 
Since
\[
\norm{f_{t h}}_{L^p(\Nc)}= \norm*{(\zeta,x)\mapsto \norm{f^{(\zeta,x,h)}_t}_{L^p(\R)}}_{L^p(E\times (\R h)^\perp)}
\]
for every $t\in \R$, this proves that
\[
\norm{f_{t h}}_{L^p(\Nc)}\meg   \ee^{H_{K_{\eps,h}}(t)} \norm{f_0}_{L^p(\Nc)}=
\ee^{H_{K_\eps}(t h)} \norm{f_0}_{L^p(\Nc)} 
\]
for every $t\in \R$. By the arbitrariness of $\eps>0$ and $h\in \partial B_F(0,1)$, the   assertion follows.
\end{proof}

\begin{oss}\label{oss:1}
We observe explicitly that, in the course of the proof of Theorem~\ref{prop:6}, we actually gave a stronger characterization of $\Bc^p_K(\Nc)$.  Namely,  $\Bc^p_K(\Nc)$ may be equivalently defined as the set of $f\in \Hol(E\times F_\C)$ such that $f_0\in L^p(\Nc)$ and such that for every $h\in F$, for every $\eps>0$, and for almost every $(\zeta,x)\in E\times (\R h)^\perp$ there is $C_{\eps,h,\zeta,x}>0$ such that 
\[
\abs{f(\zeta,x+i\Phi(\zeta)+w h)}
\meg C_{\eps,h,\zeta,x} \ee^{C_{\eps,h,\zeta,x}\abs{\Re w}+H_{K+\overline B_{F'}(0,\eps)}((\Im w) h)}
\]
for every $w\in \C$. 
\end{oss}

\begin{teo}\label{prop:7}
Let $K$ be a  compact convex subset of
$F'$, and take $p\in (0,\infty)$. Then, the following hold:
\begin{enumerate}
\item[\textnormal{(1)}] there is a constant $C>0$ such that  $\norm{f_h}_{L^q(\Nc)}\meg \frac{C}{\abs{h}^{1/p-1/q}} \ee^{H_K(\rho(\zeta,z))}\norm{f}_{\Bc^p_K(\Nc)}$ for every non-zero $h\in F$, for every $q\in [p,\infty]$, and for every $f\in \Bc^p_K(\Nc)$;

\item[\textnormal{(2)}] $\lim_{(\zeta,z)\to \infty} \ee^{-H_K(\rho(\zeta,z))} f(\zeta,z)=0$ for every $f\in \Bc^p_K(\Nc)$.
\end{enumerate}
\end{teo}

In particular, (1) ensures that $\Bc^p_K(\Nc)$ embeds continuously into $\Bc^q_K(\Nc)$.

This result extends classical, significant properties of functions in the Bernstein spaces $\Bc^p_\tau(\R)$, see~\cite[Section 20.1]{Levin-Lectures}. Observe, in fact, that  (2) implies that $f_h(\zeta,x)\to 0$ for $(\zeta,x)\in \Nc$, uniformly as $h$ runs through a compact subset of $F$, and that $f_h(\zeta,x)=o(\ee^{H_K(h)})$ for $h\to \infty$, for every $(\zeta,x)\in \Nc$.

\begin{proof}
(1) Notice that, by H\"older's inequality and Theorem~\ref{prop:6}, it will suffice to prove the case $q=\infty$. We may also assume that $K\neq \emptyset$.
Assume first that $E=\Set{0} $, $F=F'=\R$, and $K=[a,b]$ for some $a\meg b$, so that
\[
H_K(h)= \begin{cases}
-a h & \text{if $h\Meg 0$}\\
-b h & \text{if $h\meg 0$}
\end{cases}
\] 
for every $h\in \R$.  Define
\[
\Kc\colon \C \ni z \mapsto \Fc^{-1}(\chi_{[a,b]} \ee^{-(\,\cdot\,)\Im z})(\Re z)=\frac{\ee^{i b z}-\ee^{i a z}  }{2 \pi i z}\in \C,
\]
and observe that
\[
\norm{\Kc_y}_{L^{r}(\R)}\meg \frac{1}{\pi \abs{y}^{1/r'}}\ee^{H_K(y)} \norm{(1+\abs{\,\cdot\,}^2)^{-1/2}}_{L^{r}(\R)}
\]
for every $y\in \R$, and for every $r>1$. Then, set $C_r\coloneqq  \frac{1}{\pi } \norm{(1+\abs{\,\cdot\,}^2)^{-1/2}}_{L^{r}(\R)}$. Take $f\in \Bc^p_K(\R)$, and observe that by Theorem~\ref{prop:6} and~\cite[Theorem 7.3.1]{Hormander}, the Fourier transform of $f_0$ is a Radon measure supported in $K$,\footnote{Observe that the proof of Theorem~\ref{prop:6} shows that $f_0\in L^\infty(\R)$, so that the Fourier transform of $f_0$ is well defined.}. 
Therefore, 
\[
f_y=\Fc^{-1}(\ee^{-(\,\cdot\,) y}\Fc f_0) =\Fc^{-1}(\Fc f_0 \Fc \Kc_y) = f_0*\Kc_y
\]
for every $y\in \R$, so that
\[
\norm{f_y}_{L^\infty}\meg C_{p'} \frac{\ee^{H_K(y)}}{\abs{y}^{1/p}}
\]
for every $y\in \R$. 

The general case follows applying the previous case to the functions $\C\ni w\mapsto f(\zeta,x+ w h)\in \C$ for every $h\in \partial B_F(0,1)$ and for almost every $(\zeta,x)\in E\times (\R h)^\perp$, and then applying Fubini's theorem.

(2) By (1), it will suffice to prove that $f_h(\zeta,x)\to 0$ for $(\zeta,x)\to \infty$, uniformly as $h\in B_F(0,r)$, for every $r>0$. Then, take  $\phi \in \widetilde \Sc_{\overline B_{F'}(0,1)\cap \overline{\Lambda_+}}(\Nc)$  so that $\phi(0,0)=1$, and choose $g\in\Bc^\infty_{\overline B_{F'}(0,1)}(\Nc)$ so that $g_0=\phi$, so that~\cite[Theorem 4.2]{Calzi3} implies that the set of the $g_h$, for $h\in B_F(0,r)$, is bounded in $\Sc(\Nc)$. Define $g^{(j)}\coloneqq g(2^{j/2}\,\cdot\,)$  for every $j\in \N$, so that $g^{(j)}\in  \Bc^\infty_{\overline B_{F'}(0,2^{-j})}(\Nc)$. Then, $f g^{(j)}$ converges to $f$ in $\Bc^p_{K+\overline B_{F'}(0,2^{-j})}(\Nc)$, hence in $\Bc^\infty_{K+\overline B_{F'}(0,2^{-j})}(\Nc)$ by (1). The assertion follows by means of Theorem~\ref{prop:6}.
\end{proof}

\begin{cor}\label{cor:2}
Let $K$ be a compact convex subset of $F'$, and take $p\in
(0,\infty]$. Then, $\Bc^p_K(\Nc)$ is a quasi-Banach space and embeds continuously into
$\Hol(E\times F_\C)$. 
\end{cor}

\begin{proof}
Since clearly $\Bc^\infty_K(\Nc)$ embeds continuously into
$\Hol(E\times F_\C)$, Theorem~\ref{prop:7} implies that the same
holds for $\Bc^p_K(\Nc)$.  
Now, let $(f^{(j)})$ be a Cauchy sequence in $\Bc^p_K(\Nc)$. Then,
$(f^{(j)})$ converges to some $f$ in $\Hol(E\times F_\C)$ such that
$f_0\in L^p(\Nc)$. In addition, Theorem~\ref{prop:6} and
Theorem~\ref{prop:7} show that there is a constant $C>0$ such that 
\[
\norm{f_h}_{L^\infty(\Nc)}\meg C  \ee^{H_{K}(h) }  \norm{f_0}_{L^p(\Nc)}
\]
for every $h\in F$. 
Hence, $f\in \Bc^p_K(\Nc)$. Therefore, $(f^{(j)})$ converges to $f$ in $\Bc^p_K(\Nc)$.
\end{proof}

\begin{proof}[Proof of Theorem~\ref{cor:1}.]
(1) Take $f\in \Bc^p_K(\Nc)$. Then, Theorem~\ref{prop:6} and Proposition~\ref{prop:1} imply that $f_0\in \Oc_K(\Nc)\cap L^p(\Nc)$. Conversely, if $\phi\in \Oc_K(\Nc)\cap L^p(\Nc)$, then Proposition~\ref{prop:1} implies that there is a unique $f\in \Hol(E\times F_\C)$ such that $f_0=\phi$, and that $f\in \widetilde \Ec_K(E\times F_\C,\Nc)$. Therefore, $f\in \Bc^p_K(\Nc)$. 

(2)--(3) The assertions follow from (1) and~\cite[Proposition 5.7 and Corollaries 5.9 and 5.10]{Calzi3}.
\end{proof}

\begin{prop}\label{prop:3}
Take $p\in (0,\infty)$ and a  compact convex subset $K$ of $F'$.
Then, $\Bc_K^p(\Nc)\neq \Set{0}$ if and only if $K\cap \overline{\Lambda_+}$ has a non-empty interior.
In addition, $\Bc^\infty_K(\Nc)\neq \Set{0}$ if and only if $K\cap \Phi(E)^\circ\neq \emptyset$.
\end{prop}

\begin{proof}
The first assertion follows from  Theorem~\ref{cor:1}  and~\cite[Corollary 5.9]{Calzi3}.  Since $\Bc^\infty_K(\Nc)=\Bc^{\infty}_{K\cap \Phi(E)^\circ}(\Nc)$ by  Theorem~\ref{cor:1}, in order to conclude it will suffice to observe that
the function $(\zeta,z)\mapsto \ee^{i \langle\lambda_\C, z\rangle}$ belongs to $B^\infty_{K}(\Nc)$ for every $\lambda\in K\cap \Phi(E)^\circ$.  
\end{proof}

\begin{prop}\label{prop:11}
Take $p\in (0,\infty]$, $\lambda\in F'$, and a  compact convex subset $K$ of $F'$ such that $\Bc_K^p(\Nc)\neq \Set{0}$. Then, multiplication by the function $(\zeta,z)\mapsto \ee^{i\langle \lambda_\C, z \rangle}$  induces a continuous linear mapping $\Bc^p_K(\Nc)\to \Bc^p_{K+\lambda}(\Nc)$ if and only in $\lambda\in \Phi(E)^\circ$.
\end{prop}

Recall that $\Phi(E)^\circ=\overline{\Lambda_+}$ when $\Nc$ is a Siegel CR manifold (cf.~Remark~\ref{CR-remark}).

\begin{proof}
Define $\psi\colon (\zeta,z)\mapsto \ee^{i\langle \lambda_\C, z \rangle}$, and observe that $\psi$ is bounded if and only if $\lambda\in \Phi(E)^\circ$. Hence, if $\lambda\in \Phi(E)^\circ$, then multiplication by $\psi$ induces a continuous linear mapping $\Bc^p_K(\Nc)\to \Bc^p_{K+\lambda}(\Nc)$. 

Conversely, assume that  $\lambda\not \in \Phi(E)^\circ$. Fix a non-zero $f\in \Bc^p_K(\Nc)$, and observe that, by the translation invariance of $\Bc^p_K(\Nc)$ (and of its norm), the $f((\zeta,x)\,\cdot\,)$, for $(\zeta,x)\in \Nc$, form a bounded subset of $\Bc^p_K(\Nc)$. In particular, we may assume that there is $\eps>0$ such that
\[
\abs{f_0(\zeta,x)}\Meg \eps
\]
for every $(\zeta,x)\in B((0,0),\eps)$. In addition, by assumption, there is $\zeta\in E$ such that $\langle \lambda, \Phi(\zeta)\rangle=-1$. Observe that, by homogeneity, there is a constant $C>1$ such that
\[
\frac{1}{C}(\abs{\zeta}+\abs{x}^{1/2})\meg d((0,0),(\zeta,x))\meg C(\abs{\zeta}+\abs{x}^{1/2})
\]
for every $(\zeta,x)\in \Nc$, and that we may assume that $\eps>0$ is so small that
\[
\langle \lambda, \Phi(\zeta'+t \zeta)\rangle=-t^2+\langle \lambda, \Phi(\zeta')\rangle +2 t\Re\langle \lambda,\Phi(\zeta,\zeta')\rangle \meg -t^2/2
\]
for every $\zeta'\in B_E(0, \eps/C)$ and for every $t\Meg 1$.
Then,
\[
\norm{\psi f((t\zeta,0)^{-1}\,\cdot\,)}_{\Bc^p_K(\Nc)}\Meg\eps\norm{\chi_{B((t\zeta,0),\eps)} \ee^{-\langle \lambda, \Phi\circ \pr_E\rangle} }_{L^p(\Nc)}\Meg \eps \ee^{t^2/2} \norm{\chi_{B((t\zeta,0),\eps)}}_{L^p(\Nc)}
\]
for every $t\Meg 1$. Therefore, multiplication by $\psi$ does \emph{not} induce a continuous linear mapping $\Bc^p_K(\Nc)\to \Bc^p_{K+\lambda}(\Nc)$.
\end{proof}

\section{The Bernstein Inequality and a Characterization of the Bernstein Spaces}\label{sec:4}

We now prove Theorem \ref{Bern-ineq-thm} and Proposition~\ref{prop:9}. In Proposition~\ref{comp-tang-dir:Prop} we prove a weaker version dealing
with complex tangential derivatives.

\begin{proof}[Proof of Theorem \ref{Bern-ineq-thm}.]
\textsc{Step I.} Notice that, arguing by induction, it will suffice to
prove the first assertion when $k=1$.  
Take $f\in \Bc_K^p(\Nc)$ and a non-zero  $v\in F$. Arguing as in the proof of Theorem~\ref{prop:6}, we see that there is
a symmetric bounded closed interval $K_v=[-a_v,a_v]$ in $\R$, with
$a_v\Meg 0$, such that 
\[
H_{K}(t v)=H_{K_v}(t)=a_v\abs{t}
\]
for every $t\in \R$.
For every $(\zeta,x)\in E\times (\R v)^\perp$, define
\[
f^{(\zeta,x)}\colon \C\ni w\mapsto f(\zeta,x+i \Phi(\zeta)+w v)\in \C,
\]
so that $f^{(\zeta,x)}\in \Bc^p_{K_v}(\R)$ for almost every
$(\zeta,x)\in E\times (\R v)^\perp$. Then,
the classical Bernstein inequality, see e.g.~\cite[Theorem 4]{Andersen}, shows that
\[
\norm{ (f^{(\zeta,x)}_0)'}_{L^p(\R)}\meg a_v\norm{f^{(\zeta,x)}_0}_{L^p(\R)}
=H_K(v)\norm{ f^{(\zeta,x)}_0}_{L^p(\R)} 
\]
for almost every $(\zeta,x)\in E\times (\R v)^\perp$. Since 
\[
\norm{\partial_v
  f_0}_{L^p(\Nc)}=\abs{v}^{-1/p}\norm{(\zeta,x)\mapsto \norm{D
    f^{(\zeta,x)}_0}_{L^p(\R)}}_{L^p(E\times (\R v)^\perp)} 
\]
for every $k\in \N$, this implies that
\[
\norm{\partial_v f_0}_{L^p(\Nc)}\meg H_{K}(v) \norm{f_0}_{L^p(\Nc)}.
\]

\textsc{Step II.} Take a CR $\phi\in L^p(\Nc)$ such that $\partial_v^k \phi\in L^p(\Nc)$ for every $v\in \partial B_F(0,1)$ and for every $k\in \N$, and such that
\[
\limsup_{k\to \infty}\norm{\partial_v^k \phi}_{L^p(\Nc)}^{1/k}\meg H_{K}(v)
\]
for every $v\in \partial B_F(0,1)$. For every $v\in \partial B_F(0,1)$ and for every $\eps>0$, take $C_{1,v,\eps}>0$ such that
\[
\norm{\partial_v^k \phi}_{L^p(\Nc)}\meg   C_{1,v,\eps}
(H_{K}(v)+\eps)^k 
\]
for every $k\in \N$.

Take $\tau\in \widetilde\Sc_{\overline{\Lambda_+}\cap B_{F'}(0,1)}$ such that
$\tau(0,0)=1$, a non-negative $\eta\in \Sc(\Nc)$ such that $\norm{\eta}_{L^1(\Nc)}=1$, and define $\tau_j\coloneqq\tau(2^{-j}\,\cdot\,)$ and $\eta_j\coloneqq 2^{Q j}\eta(2^j\,\cdot\,)$ for every $j\in\N$. 
Observe that $\eta_j*( \phi\tau_j)$ is CR and belongs to $\Sc(\Nc)$.  
Set $\widetilde p=\min(2,p)$, and choose $\widetilde q,s\in [1,\infty]$ so that
\[
\frac{1}{\widetilde q}+\frac{1}{\widetilde p} = 1 +\frac12 \qquad \text{and} \qquad \frac{1}{\widetilde p}=\frac{1}{p}+\frac{1}{s}.
\]
By Young's and H\"older's inequalities, 
\[
\begin{split}
\limsup_{k\to \infty} \norm{\partial_v^k [\eta_j*( \phi\tau_j)]}_{L^2(\Nc)}^{1/k}
&\meg \limsup_{k\to \infty}  \left( \norm{\eta_j}_{L^{\widetilde q}(\Nc)}\norm{\partial_v^k (\phi \tau_j)}_{L^{\widetilde p}(\Nc)} \right)^{1/k}\\
&\meg \limsup_{k\to \infty} \left( \sum_{h=0}^k \binom{h}{k}  \norm{\partial_v^h \phi}_{L^p(\Nc)} \norm{\partial_v^{k-h}\tau_j}_{L^s(\Nc)}\right) ^{1/k}\\
&\meg \limsup_{k\to \infty} C_{1,v,\eps}^{1/k}\norm{\tau}_{L^s(\Nc)}^{1/k} \left(\sum_{h=0}^k \binom{h}{k}  (H_{K}(v)+\eps)^h 2^{-j(k-h)}\right) ^{1/k}\\
&=H_{K}(v)+\eps+2^{-j}
\end{split}
\]
for every $\eps>0$ and for every $j\in \N$.
By the arbitrariness of $\eps>0$, this implies that 
\[
\limsup_{k\to \infty} \norm{\partial_v^k [\eta_j*( \phi \tau_j)]}_{L^2(\Nc)}^{1/k}\meg H_{K}(v)+2^{-j}
\]
for every $j\in \N$.
By Proposition~\ref{prop:4}, this is equivalent to saying that
\[
\limsup_{k\to \infty} \left(\int_{\Lambda_+}
  \norm{\pi_\lambda(\eta_j*( \phi \tau_j))}_{\Lin^2(\sH_\lambda)}
  \abs{\la \lambda,v\ra}^k \abs{\Pfaff(\lambda)}\,\dd
  \lambda    \right)^{1/k}\meg H_{K}(v)+2^{-j}=H_{K_j}(v) 
\]
for every $v\in \partial B_F(0,1)$, where $K_j\coloneqq K+\overline B_{F'}(0, 2^{-j})$. Therefore, $\pi_\lambda(\eta_j*(\phi \tau_j))=0$ for almost every $\lambda\in\Lambda_+\setminus K_j$, since 
\[
K_j=\Set{\lambda'\in F'\colon \forall v\in \partial
  B_F(0,1)\quad \abs{\la \lambda',v \ra} \meg H_{K_j}(v)} 
\]
by the symmetry of $K_j$ (cf.~\cite[Theorem
4.3.2]{Hormander}). Therefore, for every $j\in \N$ there is a unique
$f^{(j)}\in \Bc^2_{K_j}(\Nc)$
such that $f^{(j)}_0=\eta_j*( \phi \tau_j)$, whence, in particular, $f^{(j)}\in \Bc^p_{K_j}(\Nc)$. Since clearly
$\eta_j*(\phi\tau_j)$ converges to $\phi$ in $L^p(\Nc)$,
Corollary~\ref{cor:2} implies that there is a unique 
\[
f\in \bigcap_{j\in \N} \Bc_{ K_j}^p(\Nc)=\Bc_{K}^p(\Nc)
\]
such that $f_0=\phi$.  
\end{proof}

\begin{proof}[Proof of Proposition~\ref{prop:9}.]
\textsc{Step I.} Let us first prove that
\[
\liminf_{k\to \infty}\norm{P(-i \partial_F)^k f}_{L^p(\Nc)}^{1/k}\Meg \max_{\Sp(f)} \abs{P}.
\]
Observe that, since $f$ has bounded spectrum, by Theorem~\ref{prop:7}, we  may reduce to the case $p=\infty$. Then,~\cite[Proposition 2.4]{AndersenDejeu} implies that
\[
\begin{split}
\liminf_{k\to \infty}\norm{P(-i \partial_F)^k f}_{L^\infty(\Nc)}^{1/k}&\Meg \sup_{\zeta\in E} \liminf_{k\to \infty}\norm{P(-i \partial_F)^k[f(\zeta,\,\cdot\,)]}_{L^\infty(F)}^{1/k}\Meg \sup_{\zeta\in E} \max_{\supp(\Fc_F(f(\zeta,\,\cdot\,)))} \abs{P}=\max_{\Sp(f)} \abs{P}.
\end{split}
\]

\textsc{Step II.} Now, let us prove that
\[
\limsup_{k\to \infty} \norm{P(-i \partial_F)^k f}_{L^p(\Nc)}^{1/k}\meg \max_{\Sp(f)} \abs{P}.
\]
Fix $N\in \N$, $N>m(\frac1p-\frac 1 2)$, $\eps>0$, and $\tau\in C^\infty_c(F')$ so that $\chi_{\Sp(f)}\meg \tau\meg \chi_{\Sp(f)+B_{F'}(0,\eps)} $. Define
\[
C_1\coloneqq \sup_{h=0,\dots, N}\norm{D^{(h)}\tau}_{L^\infty(F')},
\]
where $D^{(h)}\tau $ denotes the $h$-th differential of $\tau$.
Observe that Leibniz's rule and Faà di Bruno's formula show that
\[
\begin{split}
\abs{ D^{(h)} (P^k \tau)(\lambda)}&\meg C_1 \chi_{\Sp(f)+ \overline B_{F'}(0,\eps)}(\lambda) \sum_{j=0}^h \binom{h}{j} \abs{D^{(j)}P^k(\lambda)}\\
&\meg C_1\chi_{\Sp(f)+\overline B_{F'}(0,\eps)}(\lambda)\sum_{j=0}^h \sum_{\sum_{\ell=1}^j \ell \gamma_\ell=j, \abs{\gamma}\meg k} \frac{h! k! }{(h-j)!\gamma!(k-\abs{\gamma})!} \abs*{P^{k-\abs{\gamma}}(\lambda) \prod_{\ell=1}^j \left( \frac{D^{(\ell)}P(\lambda)}{\ell!} \right)^{\gamma_\ell}}
\end{split}
\]
for every $\lambda\in F'$, for every $k\in \N$, and for every $h=0,\dots, N$. Therefore, there is a constant $C_2>0$ such that
\[
\norm{P^k \tau}_{W^{N,2}(F')}\meg C_2 \max_{h=0,\min(N,k)} \max_{\Sp(f)+\overline B_{F'}(0,\eps)} \abs{P}^{k-h}
\]
for every $k\in \N$, where $W^{N,2}(F')$ denotes the usual Sobolev space of type $L^2$ and order $N$. Therefore,~\cite[Theorem 1.5.2]{Triebel} implies that there is a constant $C_3>0$ such that
\[
\norm{P(-i \partial)^k f(\zeta,\,\cdot\,)}_{L^p(F)}= \norm{[f(\zeta,\,\cdot\,)]*\Fc_F^{-1}(P^k \tau)}_{L^p(F)}\meg C_3 \norm{f(\zeta,\,\cdot\,)}_{L^p(F)} \max_{h=0,\min(N,k)} \max_{\Sp(f)+\overline B_{F'}(0,\eps)} \abs{P}^{k-h},
\]
for almost every $\zeta\in E$ and for every $k\in \N$,so that
\[
\norm{P(-i \partial)^k f}_{L^p(\Nc)}^{1/k}\meg (C_3\norm{f}_{L^p(\Nc)})^{1/k}  \max_{h=0,\min(N,k)} \max_{\Sp(f)+\overline B_{F'}(0,\eps)} \abs{P}^{1-h/k},
\]
for every $k\in \N$. Therefore,
\[
\limsup_{k\to \infty} \norm{P(-i \partial_F)^k f}_{L^p(\Nc)}^{1/k}\meg \max_{\Sp(f)+\overline B_{F'}(0,\eps)} \abs{P}.
\]
By the arbitrariness of $\eps$ and the continuity of $P$, the assertion follows.
\end{proof}

\medskip

 We now consider the case of derivatives in the complex tangential
 directions of $\Nc$, that is, the directions along $E$.
We denote by $\vect S ^k(E, \Phi_\lambda)$ the $k$-th symmetric
hilbertian power of $E$, endowed with the scalar product
$\Phi_\lambda=\langle \lambda_\C, \Phi\rangle$, for every $\lambda\in
\Lambda_+$ (cf.~\cite[Chapter V, \S\ 3, No.\ 3]{BourbakiTVS}). Then, 
\[
\norm{v_1\cdots v_k}_{\vect S ^k(E, \Phi_\lambda)}^2= \sum_{\sigma\in
  \Sf_k} \prod_{j=1}^k \Phi_\lambda(v_j,v_{\sigma(j)}) 
\] 
for every $v_1,\dots, v_k\in E$, where $\Sf_k$ denotes the set of
permutations of $\Set{1,\dots, k}$. 

\begin{prop}\label{comp-tang-dir:Prop}
    Let $K$ be a  compact convex subset of $\overline{\Lambda_+}$. 
    Then, for every $f\in \Bc_K^2(\Nc)$, for every $k\in \N$, and for every $v_1,\dots,v_k\in E$,
    \[
    \norm{Z_{v_1}\cdots Z_{v_k} f_0}_{L^2(\Nc)}\meg  2^{k/2}  \sup_{\lambda\in K\cap \Lambda_+} \norm{v_1\cdots v_k}_{\vect S ^k(E, \Phi_\lambda)} \norm{f_0}_{L^2(\Nc)}.
    \]
    
    Conversely, assume that $\Phi(E)$ generates $F$ as a vector space, and define 
    \[
    K'\coloneqq \Set{\lambda\in \overline{\Lambda_+}\colon \forall \zeta\in E\:\: \langle \lambda, \Phi(\zeta)\rangle\meg \abs{H_K(\Phi(\zeta))} },
    \]
    so that $K'$ is a compact convex subset of $F'$. In addition, take $p\in [1,\infty]$ and a CR $\phi\in L^p(\Nc)$ such that $Z_v^k \phi\in L^p(\Nc)$ for every $v\in \partial B_E(0,1)$ and for every $k\in \N$, and
    \[
    \limsup_{k\to \infty} \left( \frac{\norm{Z_v^k f}_{L^p(\Nc)}}{\sqrt{k!}}  \right)^{1/k}\meg \sqrt 2 \abs{H_K(\Phi(v))}
    \]
    for every $v\in \partial B_E(0,1)$. Then,  there is a unique $f\in \Bc^p_{K'}(\Nc)$ such that $f_0=\phi$.
\end{prop}

In particular, if $f\in \Bc^2_K(\Nc)$, then
\[
\norm{Z_{v}^k f_0}_{L^2(\Nc)}\meg  \sqrt{2^k  k! \abs{H_{K}(\Phi(v))}^k} \norm{f_0}_{L^2(\Nc)}
\]
for every $v\in E$ and for every $k\in \N$.

With a similar proof, one may also consider more general derivatives of $f$, where both operators of the form $Z_v$ and $\overline{Z_{v'}}$ appear. Nonetheless, it does not seem that a general, relatively simple, formula may arise.

We do not know if an analogue of the first assertion holds for $p\neq 2$.
In addition, we observe explicitly that  $Z_{v_1}\cdots Z_{v_k} f$ is not holomorphic, unless it vanishes identically or $k=0$, so that it is not possible to argue by induction in this case.

\begin{proof}
\textsc{Step I.}  Take  $v_1,\dots, v_k\in E$, and $f\in \Bc_K^2(\Nc)$, and observe that
 \[
 \begin{split}
 \pi_\lambda(Z_{v_1}\cdots Z_{v_k} f_0)&= \pi_\lambda(f_0) P_{\lambda,0} \dd\pi_\lambda(Z_{v_k})\cdots \dd \pi_\lambda(Z_{v_1})\\
  &= \langle \dd\pi_\lambda(Z_{v_k})\cdots \dd \pi_\lambda(Z_{v_1})\,\cdot\, \big\vert e_{\lambda,0}\rangle \pi_\lambda(f_0)e_{\lambda,0}\\
  &=(-1)^k\langle \,\cdot\, \big\vert  \dd\pi_\lambda(\overline{Z_{v_1}})\cdots \dd \pi_\lambda(\overline{Z_{v_k}})e_{\lambda,0}\rangle \pi_\lambda(f_0)e_{\lambda,0}
 \end{split}
 \]
 on $C^\infty(\sH_\lambda)$, for almost every $\lambda\in \Lambda_+$. Therefore,
 \begin{align}
 \norm{ \pi_\lambda(Z_{v_1}\cdots Z_{v_k}
   f_0)}_{\Lin^2(\sH_\lambda)}^2&=
 \norm{\pi_\lambda(f_0)e_{\lambda,0}}_{\sH_\lambda}^2
 \norm{\dd\pi_\lambda(\overline{Z_{v_1}})\cdots \dd \pi_\lambda(\overline{Z_{v_k}}) e_{\lambda,0}}^2_{\sH_\lambda} \notag\\
  &=\norm{\pi_\lambda(f_0)}_{\Lin^2(\sH_\lambda)}^2\norm{\dd\pi_\lambda(\overline{Z_{v_1}})\cdots
    \dd \pi_\lambda(\overline{Z_{v_k}}) e_{\lambda,0}}^2_{\sH_\lambda} \label{eq-with-label:1}
\end{align}
 for almost every $\lambda\in \Lambda_+$. Now, take $\lambda\in \Lambda_+$, and observe that 
 \[
 \begin{split}
  \norm{\dd\pi_\lambda(\overline{Z_{v_1}})\cdots \dd \pi_\lambda(\overline{Z_{v_k}}) e_{\lambda,0}}^2_{\sH_\lambda}&= (-1)^k\langle \dd\pi_\lambda(Z_{v_k})\cdots\dd \pi_\lambda(Z_{v_1})\dd\pi_\lambda(\overline{Z_{v_1}})\cdots \dd \pi_\lambda(\overline{Z_{v_k}}) e_{\lambda,0}\vert e_{\lambda,0}\rangle\\
&=2^k \partial_{E,v_k}\cdots \partial_{E,v_1}(\Phi_\lambda(\,\cdot\,,v_1)\cdots \Phi_\lambda(\,\cdot\,,v_k))(0)\\
&=2^k\sum_{\sigma\in \Sf_k} \prod_{j=1}^k \Phi_\lambda(v_j,v_{\sigma (j)})\\
&=2^k \norm{v_1\cdots v_k}_{\vect S^k(E,\Phi_\lambda)}^2,
 \end{split}
 \]
 since 
 \[
 \langle \psi \vert e_{\lambda,0}\rangle= \sqrt{\frac{\pi^n}{2^n \abs{\Pfaff(\lambda)}}} \psi(0)
 \]
 and $e_{\lambda,0}=\sqrt{\frac{2^n \abs{\Pfaff(\lambda)}}{\pi^n}}\chi_E$  for every $\psi\in \sH_\lambda$.
 
 Therefore, Proposition~\ref{prop:4} shows that
 \[
 \begin{split}
  \norm{Z_{v_1}\cdots Z_{v_k} f_0}_{L^2(\Nc)}^2&=\frac{2^{n-m}}{\pi^{n+m}} \int_{K} \norm{\pi_\lambda(Z_{v_1}\cdots Z_{v_k} f_0)}^2_{\Lin^2(\sH_\lambda)}\abs{\Pfaff(\lambda)}\,\dd \lambda\\
&=\frac{2^{n-m}}{\pi^{n+m}} 2^k \int_{K}  \norm{v_1\cdots v_k}_{\vect S^k(E,\Phi_\lambda)}^2 \norm{\pi_\lambda( f_0)}^2_{\Lin^2(\sH_\lambda)}\abs{\Pfaff(\lambda)}\,\dd \lambda\\
&\meg  2^k \sup_{\lambda\in K\cap \Lambda_+}  \norm{v_1\cdots v_k}_{\vect S^k(E,\Phi_\lambda)}^2 \norm{f_0}_{L^2(\Nc)}^2,
 \end{split}
 \]
 whence the result.
 
 \textsc{Step II.} Now, assume that $\Phi(E)$ generates $F$ as a vector space, and observe that $K'$ is a compact convex subset of $\overline{\Lambda_+}$, since by definition its orthogonal projections  into the lines $(\R\Phi(v))^{\circ \perp}$, $v\in E$, are bounded.
 Take $p\in [1,\infty]$ and a CR $\phi\in L^p(\Nc)$ such that $Z_v^k \phi\in L^p(\Nc)$ for every $v\in \partial B_E(0,1)$ and for every $k\in \N$, and
 \[
 \limsup_{k\to \infty} \left( \frac{\norm{Z_v^k f}_{L^p(\Nc)}}{\sqrt{k!}}  \right)^{1/k}\meg \sqrt 2 \abs{H_K(\Phi(v))}
 \]
 for every $v\in \partial B_E(0,1)$. 
 Arguing by approximation as in the proof of Theorem~\ref{Bern-ineq-thm}, we may reduce to the case $p=2$, in which case, by~\textsc{step I}, 
 \[
 \limsup_{k\to \infty} \left(  \int_{\Lambda_+} \langle \lambda, \Phi(v)\rangle ^k\norm{\pi_\lambda( \phi)}^2_{\Lin^2(\sH_\lambda)}\abs{\Pfaff(\lambda)}\,\dd \lambda\right)^{ 1/(2k)} \meg \abs{H_K(\Phi(v))}
 \]
 for every $v\in \partial B_E(0,1)$, hence for every $v\in E$. Then, $\pi_\lambda(\phi)=0$ for almost every $\lambda\in \Lambda_+\setminus K'$, so that $\phi\in \Oc_{K'}(\Nc)\cap L^2(\Nc)$. The assertion follows from  Theorem~\ref{cor:1}. 
\end{proof}

\medskip

\section{Boundedness of the orthogonal projections}\label{sec:5}

In this section we study the boundedness of the orthogonal projection
$P\colon L^2(\Nc)\to \Bc_K^2(\Nc)$ on the spaces $L^p(\Nc)$, that is, for
which $p\in[1,\infty)$ the restriction of $P$ to $L^p(\Nc) \cap
L^2(\Nc)$ extends to a bounded operator in $L^p(\Nc)$.  We begin by
describing the integral kernel of $P$.

\begin{prop}\label{prop:12}
Let $K$ be a  compact convex subset of
$\overline{\Lambda_+}$. Then, the reproducing kernel of
$\Bc^{2}_K(\Nc)$ is given by 
\[
\Kc((\zeta,z),(\zeta',z'))\coloneqq \frac{2^{n-m} }{\pi^{n+m}} \int_K \ee^{\la \lambda,  i(z-\overline{z'})+2\Phi(\zeta,\zeta') \ra }\abs{\Pfaff(\lambda)}\,\dd \lambda=\Lc(\chi_K\abs{\Pfaff(\,\cdot\,)})\left(\frac{z-\overline{z'}}{i}-2\Phi(\zeta,\zeta')  \right) 
\]
for every $(\zeta,z),(\zeta',z')\in E\times F_\C$. 
\end{prop}

Observe that the self-adjoint projector $P$ of $L^2(\Nc)$ onto $L^2(\Nc)\cap \Oc_K(\Nc)$ is a right convolution operator with (convolution) kernel $\Kc(\,\cdot\,,(0,0))_0 $. Since this latter function is \emph{not a bounded measure}  (for $\Fc_F \Kc((\zeta,\,\cdot\,),(0,0))_0$ is \emph{not} continuous for any $\zeta\in E$), $P$ cannot induce endomorphisms of $L^1(\Nc)$ and $L^\infty(\Nc)$ (cf., e.g.,~\cite[Theorem 35.5]{HR} for $L^1(\Nc)$ and argue by transposition for $L^\infty(\Nc)$). 

\begin{proof}
It suffices to observe that, necessarily, $\Kc(\,\cdot\,,(\zeta',z'))_0$ is a CR element of $L^2(\Nc)$ and
\[
\pi_\lambda(\Kc(\,\cdot\,,(\zeta',z'))_0)=\chi_K(\lambda) \pi_\lambda(\zeta',\Re z') \ee^{-\la \lambda, \rho(\zeta',z')\ra}
\]
for every $(\zeta',z')\in E\times F_\C$ and for almost every $\lambda\in \Lambda_+$, and to argue as in the proof of~\cite[Corollary 1.42]{CalziPeloso}. 
\end{proof}

Notice that, in general, $\chi_K$ is \emph{not} a Fourier multiplier of $L^p$,
$p\neq 2$, so that, in general, one cannot expect the self-adjoint
projector of $L^2(\Nc)$ onto $\Bc^2_K(\Nc)$ to induce a continuous
linear projector of $L^p(\Nc)$ onto $\Bc^p_K(\Nc)$ for any $p\neq 2$. 
For the same reason, it is not clear if the reproducing kernel of
$\Bc^2_K(\Nc)$ reproduces the elements of $B^p_K(\Nc)$, $p> 2$ (or
even if it makes sense to ask for such a reproducing property). 

\begin{prop}\label{prop:2}
Let $K$ be a  compact convex subset of $\Lambda_+$, and take
$p\in  (1,\infty)$. Denote by $P$ the self-adjoint projector of
$L^2(\Nc)$ onto $L^2(\Nc)\cap \Oc_K(\Nc)$. Then, $P$ induces an
endomorphism of $L^p(\Nc)$ if and only if $\Fc_F^{-1}(\chi_K)$ is a
convolutor of $L^p(F)$.  
\end{prop}

Notice that we do not require that $K$ be a  compact convex
subset of $\overline{\Lambda_+}$, but rather of $\Lambda_+$.

\begin{proof}
\textsc{step I.} Assume first that $P$ induces an endomorphism of
$L^p(\Nc)$, and let us prove that the operator $P'\colon f \mapsto
f*\Fc_F^{-1}(\chi_K)$ induces an endomorphism of $L^p(F)$. Observe
first that, if $\phi\in \Fc_F^{-1}(C^\infty_c(\Lambda_+))$ and
$\Fc_F \phi=1$ on $K$, then $P' f= P'(f*\phi)$ for every $f\in
\Sc(F)$.  
Now, denote by 
\[
\sE\colon \Oc_{K'}(F)\ni f\mapsto \frac{1}{(2\pi)^m}\int_{K'} \Fc_F f(\lambda) \ee^{i\langle \lambda_\C,\,\cdot\,\rangle}\,\dd \lambda \in \Hol(F_\C)
\]
the canonical extension mapping, and observe that Theorem~\ref{prop:6} implies that 
\[
\norm{[\sE (f*\phi)]_{h}}_{L^q(F)}\meg \ee^{H_{K'}(h)} \norm{f*\phi}_{L^q(F)}
\]
for every $f\in \Sc(F)$, for every $h\in F$, and for every $q\in (0,\infty]$, where $K'$ is the (closed) convex envelope of  $\Supp{\Fc_F \phi}$. Then, define 
\[
\sE_\phi f\colon E\times F_\C\ni (\zeta,z)\mapsto \sE(f*\phi)(z)\in \C
\]
for every $f\in \Sc(F)$, and observe that
\[
\norm{(\sE_\phi f)_h}_{L^q(\Nc)}\meg \ee^{H_{K'}(h)} \norm{\zeta \mapsto \ee^{H_{K'}(\Phi(\zeta))}}_{L^q(\Nc)} \norm{f*\phi}_{L^q(F)}
\]
for every $h\in F$ and for every $q\in (0,\infty]$, since $H_{K'}$ is sub-additive.
Now, observe that~\cite[Lemma 5.6]{Calzi3} implies that there is a constant $c>0$ such that
\[
\langle \lambda, \Phi(\zeta)\rangle \Meg c \abs{\Phi(\zeta)}d(\lambda, \partial \Lambda_+)
\]
for every $\lambda\in \overline{\Lambda_+}$ and for every $\zeta\in E$ (so that $\Phi(\zeta)\in \Lambda_+^\circ$). Using the fact that $d(K',\partial \Lambda_+)>0$ and that $\abs{\Phi}$ is homogeneous of degree $2$ and proper, we then see that there is a constant $c'>0$ such that
\[
H_{K'}(\Phi(\zeta))\meg -c' \abs{\zeta}^2
\] 
for every $\zeta\in E$, so  that the mapping $\zeta \mapsto \ee^{H_{K'}(\Phi(\zeta))}$ belongs to $L^q(E)$ for every $q\in (0,\infty]$. Hence, $\sE_\phi f \in \Bc_{K'}^q(\Nc)$ and there is a constant $C_{1,q}>0$ such that
\[
\norm{(\sE_\phi f)_0}_{L^q(\Nc)}\meg C_{1,q} \norm{f}_{L^q(F)}
\]
for every $f\in \Sc(\Nc)$. In particular, $\sE_\phi f\in \Bc^2_{K'}(\Nc)$, so that 
\[
\pi_\lambda(P[(\sE_\phi f)_0])= \chi_K(\lambda) \pi_\lambda((\sE_\phi f)_0)
\]
for every $f\in \Sc(F)$ and for almost every $\lambda\in \Lambda_+$. In addition, Proposition~\ref{prop:8} implies that
\[
\pi_\lambda((\sE_\phi f)_0)=\frac{\pi^n}{2^n\abs{\Pfaff(\lambda)}}(\Fc_F f)(\lambda)(\Fc_F \phi)(\lambda) 
\]
 for every $\lambda\in \Lambda_+$, so that
 \[
 P[(\sE_\phi f)_0]=\sE_\phi( f* \Fc_F^{-1}(\chi_K)  )
 \]
 for every $f\in \Sc(F)$. To conclude, observe that Cauchy's estimates imply that
 \[
 \abs{(f*\phi)(x)}\meg \dashint_{B_{E\times F_\C}((0,0),1)}\abs{\sE_\phi f(\zeta,z+x)}\,\dd (\zeta,z)
 \]
 for every $x\in F$, so that by means of Jensen's inequality and the preceding estimates we see that there is a constant $C_2>0$ such that 
 \[
 \norm{f*\phi}_{L^p(F)}\meg C_2 \norm{(\sE_\phi f)_0}_{L^p(\Nc)}
 \]
 for every $f\in \Sc(F)$, so that
 \[
 \norm{f*\Fc_F^{-1}(\chi_K)}_{L^p(F)}\meg C_{1,p} C_2 \norm{P}_{\Lin(L^p(\Nc))} \norm{f}_{L^p(F)}
 \]
 for every $f\in \Sc(F)$.   Therefore, $L^p(F)*\Fc_F^{-1}(\chi_K)\subseteq L^p(F)$ continuousy, whence the result. 

\textsc{Step II.} Assume now that $\Fc_F^{-1}(\chi_K)$ is a convolutor of $L^p(F)$. Fix $\phi \in \widetilde\Sc_{\overline{\Lambda_+}}(\Nc)$ so that $\Fc_\Nc \phi=1$ on $K$, and define $u\coloneqq \delta_0\otimes \Fc_F^{-1}(\chi_K)$. Let us prove that 
\[
P f= (f*\phi)*u
\]
for every $f\in \Sc(\Nc)$. Then, fix a positive $\tau\in C^\infty(F')$ with integral $1$, and define
\[
u_j\coloneqq \delta_0\otimes \Fc_F^{-1}(\chi_K*[2^{mj} \tau(2^j\,\cdot\,)])
\] 
for every $j\in \N$, so that $u_j$ is a bounded measure supported in $F$. In addition,
\[
\pi_\lambda(u_j)= \Fc_F (\Fc_F^{-1}(\chi_K*[2^{mj} \tau(2^j\,\cdot\,)]))(\lambda) \pi_\lambda(0,0)=(\chi_K*[2^{mj} \tau(2^j\,\cdot\,)])(\lambda) I_{\sH_\lambda}
\]
for every $\lambda\in \Lambda_+$, so that
\[
\pi_\lambda(f*\phi*u_j)= \pi_\lambda(f) P_{\lambda,0} (\Fc_\Nc\phi)(\lambda) (\chi_K*[2^{mj} \tau(2^j\,\cdot\,)])(\lambda)
\]
for every $\lambda\in \Lambda_+$.
Passing to the limit for $j\to \infty$, and observing that $f*\phi*u_j$ converges to $f*\phi*u$ in $L^2(\Nc)$, this proves that
\[
\pi_\lambda(f*\phi*u)=\chi_K(\lambda)\pi_\lambda(f) P_{\lambda,0}
\]
for almost every $\lambda\in \Lambda_+$. Since clearly $f*\phi*u=u*f*\phi$ is CR, by Theorem~\ref{prop:6} this is sufficient to prove that $P f=f*\phi*u$.

Then,
\[
\begin{split}
\norm{P f}_{L^p(\Nc)}&\meg \norm{\zeta \mapsto  \norm{[(f*\phi)(\zeta,\,\cdot\,)]*\Fc^{-1}_F(\chi_K)}_{L^p(F)}}_{L^p(E)}\\
&\meg \norm{\,\cdot\,*\Fc^{-1}_F(\chi_K)}_{\Lin(L^p(F))}\norm{f*\phi}_{L^p(\Nc)}\\
&\meg \norm{\,\cdot\,*\Fc^{-1}_F(\chi_K)}_{\Lin(L^p(F))} \norm{\phi}_{L^1(\Nc)} \norm{f}_{L^p(\Nc)}
\end{split}
\]
for every $f\in \Sc(\Nc)$. Then, $P$ induces an endomorphism of $L^p(\Nc)$.
\end{proof}

\begin{cor}\label{cor:3}
Keep the hypotheses and the notation of Proposition~\ref{prop:2}, and
assume that $K$ has a non-empty interior. Then, the following hold: 
\begin{enumerate}
\item if $K$ is polyhedral, then $P$ induces an endomorphism of
  $L^p(\Nc)$ for every $p\in (1,\infty)$; 
\item if there is a $2$-dimensional affine subspace $\pi$ of $F'$  which meets the interior of $K$, and $(\partial K)\cap \pi$ is the graph of a non-affine convex function of class $C^1$ in the vicinity of some of its points,  then $P$ does not induce  an endomorphism of $L^p(\Nc)$, for every $p\in [1,\infty]\setminus  \Set{2}$. 
\end{enumerate}
\end{cor}

This result, for $E=\Set{0}$, may be considered as \emph{folklore} and follows from the arguments used by Fefferman in~\cite{Fefferman} to solve the multiplier problem for the ball. 
Since we were not able to find a precise reference, we indicate the main steps of the proof.

\begin{proof}
By Proposition~\ref{prop:2}, we may assume that $\Nc=F$.

\textsc{Step I.} Assume first that $p\in (1,\infty)$  and  that $K$ is
polyhedral. Then, $\chi_K$ is a finite product of characteristic
functions of half-spaces, so that the assertion follows by the
continuity of the Hilbert transform on $L^p(\R)$. 

\textsc{Step II.} Assume now that $F=F'=\R^2 $ and that $\partial K$ is the graph of a non-affine convex function of class $C^1$ in the vicinity of some point.   Up to a linear change of coordinates, we may assume that, for every $\theta\in [\pi/3,2\pi/3]$, $\ee^{i \theta}$ is the normal vector to $K$ at some smooth point
$\lambda_\theta$ of  $\partial K$. By inspection of the proof of~\cite[Theorem 1]{Fefferman}, this is sufficient to conclude that $\chi_K$ is not a Fourier multiplier of $L^p(\R^2)$ for every $p\in[1,\infty]\setminus\Set{2}$. 

\textsc{Step III.} Assume that there is a $2$-dimensional affine subspace $\pi$ of $F'$ which meets the interior of $K$, and  that  $(\partial K)\cap \pi$ is the graph of a non-affine convex function of class $C^1$ in the vicinity of some point. Observe
that, by~\cite[Proposition 16 of Chapter II, \S\ 2, No.\ 6]{BourbakiTVS}, the interior of $K\cap \pi$ in $\pi$ is  the intersection of $\pi$ and the interior of $K$  (in $F'$), so that $\partial K\cap \pi$ is the boundary of $K\cap \pi$ in $\pi$. In addition, every $\lambda\in\pi\setminus (\partial K\cap \pi)$ is a Lebesgue point for $
\chi_K$, so that~\cite{Jodeit} implies that, if $\chi_K$ is a Fourier multiplier of $L^p(F)$, then $\chi_{K\cap \pi}$ is a Fourier multiplier of $L^p((\pi-\pi)^{\circ \perp})$. By \textsc{step II}, this can be the case only if $p=2$, so that the assertion follows. 
\end{proof}
\medskip

\section{Sampling sequences}\label{sec:6}

The goal of this section is to prove a sufficient condition for a sequence
in $\Nc$ to be sampling for $\Bc_K^p(\Nc)$, Theorem \ref{prop:10}.
We first show that lattices on $\Nc$ may be constructed from lattices on $E$ and $F$, though with some loss of precision. 

\begin{lem}\label{lem:2}
There are constants $C,C'>0$ such that for every $\delta>0$, for every
$R>1$, for every $(\delta^2,R^2)$-lattice $(x_j)_{j\in J}$ in $F$ and
for every $(\delta,R)$-lattice $(\zeta_{j'})_{j'\in J'}$ in
$E$,\footnote{Notice that $E$ and $F$ are hilbertian spaces, so that
  lattices thereon are well defined.} the family
$(\zeta_{j'},x_j)_{(j,j')\in J\times J'}$ is a $(C\delta,C'
R)$-lattice on $\Nc$. 
\end{lem}

\begin{proof}
Observe first that, by homogeneity, there is a constant $C_1>0$ such that
\[
\frac{1}{C_1} d((0,0),(\zeta,x))\meg \abs{\zeta}+\abs{x}^{1/2}\meg C_1 d((0,0),(\zeta,x))
\]
for every $(\zeta,x)\in \Nc$. Then, let $(x_j)_{j\in J}$ be a
$(\delta^2,R^2)$-lattice  in $F$, and let $(\zeta_{j'})_{j'\in J'}$ be
a $(\delta,R)$-lattice  in $E$. Observe that 
\[
d((\zeta_{j'},x_{j_1}), (\zeta_{j'},x_{j_2}))=d((0,0),(0,
x_{j_2}-x_{j_1}))  \Meg \frac{1}{C_1} \abs{x_{j_2}-x_{j_1}}^{1/2}\Meg
\frac{\sqrt 2 \delta}{C_1} 
\]
for every $j'\in J'$ and for every $j_1,j_2\in J$, $j_1\neq j_2$. Analogously,
\[
d((\zeta_{j'_1},x_{j_1}),
(\zeta_{j'_2},x_{j_2}))=d((0,0),(\zeta_{j'_2}-\zeta_{j'_1},
x_{j_2}-x_{j_1}+2 \Im \Phi(\zeta_{j'_1},\zeta_{j'_2}))) \Meg
\frac{1}{C_1} \abs{\zeta_{j'_2}-\zeta_{j'_1}}\Meg \frac{2 \delta}{C_1} 
\]
for every $j'_1,j'_2\in J'$, $j_1'\neq j'_2$, and for every
$j_1,j_2\in J$. Therefore, the balls $B((\zeta_{j'},x_j),
\delta/(\sqrt 2 C_1))$, $j\in J$, $j'\in J'$, are pairwise
disjoint. It will then suffice to prove that 
\[
B_E(\zeta_{j'},R\delta)\times F \subseteq \bigcup_{j\in J} B((\zeta_{j'},x_j),2 C_1 R\delta)
\]
for every $j'\in J'$.  Indeed, fix $j'\in J'$, $\zeta\in
B_E(\zeta_{j'},R\delta)$, and $x\in F$, and observe that, since
$(x_j-2 \Im \Phi(\zeta,\zeta_{j'}))_j$ is still a
$(\delta^2,R^2)$-lattice on $F$ by translation invariance, there is
$j\in J$ such that 
\[
\abs{x_j-2 \Im \Phi(\zeta,\zeta_{j'})-x}^{1/2}\meg R\delta,
\]
so that
\[
\begin{split}
d((\zeta,x),(\zeta_{j'},x_j))&=d((0,0), (\zeta_{j'}-\zeta, x_j-x+2 \Im\Phi(\zeta,\zeta_{j'})))\\
&\meg C_1(\abs{\zeta_{j'}-\zeta}+\abs{x_j-x+2 \Im\Phi(\zeta,\zeta_{j'})}^{1/2}  )\\
&\meg 2C_1 R\delta.
\end{split}
\]
The assertion follows.
\end{proof}

We now prove Theorem \ref{prop:10}.
By Lemma~\ref{lem:2}, the preceding result may be applied to families
of the form $(\zeta_{j'},x_j)_{(j',j)\in J'\times J}$, provided that
$(\zeta_{j'})$ and $(x_j)$ are sufficiently fine lattices in $E$ and
$F$, respectively.  
The sampling theorems proved in~\cite{MonguzziPelosoSalvatori}
provide {\em sharp} explicit necessary and sufficient conditions for a natural class of
lattices of the above form to be sampling, when $F=\R$ and $p=2$.
Arguing as in~\cite[Theorem 2.1]{OlevskiiUlanovskii}, one may show
that the  same sufficient conditions  essentially  work for all $p\in
(0,\infty]$. See also~\cite{OlevskiiUlanovskii} for the case
$E=\Set{0}$.   It would be of 
interest to obtain sharp conditions also in the current setting.

\begin{proof}[Proof of Theorem \ref{prop:10}.]
Observe that, by the proof of~\cite[Theorem 5.13]{Calzi3}, there is $\phi \in \Sc(\Nc)$ such that 
\[
f=f*\tau
\]
for every $f\in \Sc(\Nc)\cap \Oc_{K+\overline B_{F'}(0,2)}(\Nc)$. Then, take $\phi \in \widetilde\Sc_{\overline B_{F'}(0,1)\cap \overline{\Lambda_+}}(\Nc)$, and observe that $f \phi \in \Sc(\Nc)\cap  \Oc_{K+\overline B_{F'}(0,2)}(\Nc)$ for every $f\in \Sc(\Nc)\cap \Oc_{K+\overline B_{F'}(0,1)}(\Nc)$, so that $(f \phi)*\tau= f \phi$. Since $\phi(2^{-j}\,\cdot\,) \in\widetilde  \Sc_{\overline B_{F'}(0,1)\cap  \overline{\Lambda_+}}(\Nc)$ for every $j\in \N$, replacing $\phi$ with $\phi(2^{-j}\,\cdot\,)$ and passing to the limit for $j\to \infty$, we see that $f*\tau=f$ for every $f\in \Oc_{K+\overline B_{F'}(0,1)}(\Nc)$. 
 
 Therefore,~\cite[Proposition 9]{Calzi2} implies that there are two
 constants $C,\delta_->0$ such that for every $(\delta,R)$-lattice
 $(\zeta_j,x_j)_{j\in J}$ on $\Nc$, with $\delta>0$ and $R\in
 (1,R_0]$, relative to a left-invariant distance on $\Nc$, one has 
\[
\frac1 C \norm{f}_{L^p(\Nc)}\meg \delta^{Q/p}\norm*{\max_{\overline    B((\zeta_j,x_j), R\delta)} \abs{f}  }_{\ell^p(J)}\meg  C\norm{f}_{L^p(\Nc)} 
\]
for every $f\in \Oc_{K+\overline B_{F'}(0,1)}(\Nc)$, if $\delta\meg \delta_+$, and
\[
\frac1 C \norm{f}_{L^p(\Nc)}\meg \delta^{Q/p}\norm*{\min_{\overline    B((\zeta_j,x_j), R\delta)} \abs{f}  }_{\ell^p(J)}\meg  C\norm{f}_{L^p(\Nc)} 
\]
for every $f\in\Oc_{K+\overline B_{F'}(0,1)}(\Nc)\cap L^p(\Nc)$, if $\delta\meg \delta_-$. To conclude,  observe that there is $\psi\in \widetilde \Sc_{\overline B_{F'}(0,1)\cap\overline{\Lambda_+}}(\Nc)$ such that $\psi(0,0)=1$, so that $\psi_j\coloneqq \psi(2^{-j}\,\cdot\,) \in \widetilde \Sc_{\overline B_{F'}(0,1)\cap\overline{\Lambda_+}}(\Nc)$ for every $j\in\N$. Then,  $f\psi_j$ converges locally uniformly to $f$ on $\Nc$, $f \psi_j\in \Oc_{K+\overline B_{F'}(0,1)}(\Nc)\cap L^p(\Nc)$,   and $\abs{f \psi_j}\meg \abs{f} \norm{\psi}_{L^\infty(\Nc)}$ for every $j\in \N$. The assertion follows. 
\end{proof}

\end{document}